\documentclass[12pt, letterpaper]{article}
\usepackage{dirtytalk}
\usepackage{amsmath}
\usepackage{amsfonts}
\usepackage{amssymb}
\usepackage{amsthm}
\usepackage{breqn}
\usepackage{setspace}
\usepackage{fullpage}
\usepackage{enumitem}
\usepackage{comment}
\usepackage{hyperref}
\usepackage{bbm}
\usepackage{tikz}
\usepackage{enumitem}
\usepackage{mathtools} 
\usepackage[utf8]{inputenc}
\usepackage[english]{babel}
\usepackage{mathtools}
\usetikzlibrary{patterns,arrows,decorations.pathreplacing}
\usepackage{xcolor}

\newtheorem{lemma}{Lemma}
\newtheorem{theorem}{Theorem} 
 
\newtheorem{definition}{Definition}
\newtheorem{claim}{Claim}

\newtheorem{conjecture}{Conjecture}

\newtheorem{remark}{Remark}

\newcommand{\p}{\mathcal{P}}

\DeclareMathOperator{\ex}{ex}

\usepackage{authblk}

\usetikzlibrary{patterns}

\title{On the rainbow planar Tur\'an number of paths}
\author[1]{Ervin Gy\H{o}ri} 
\author[2]{Ryan R. Martin}
\author[3,4]{Addisu Paulos}
\author[1]{Casey~Tompkins}
\author[5,6]{Kitti Varga}

\affil[1]{Alfr\'ed R\'enyi Institute of Mathematics, Budapest}
\affil[2]{Iowa State University, Ames}
\affil[3]{E\"otv\"os Lor\'and University, Budapest}
\affil[4]{Addis Ababa University, Addis Ababa}
\affil[5]{Budapest University of Technology and Economics, Budapest}
\affil[6]{MTA-ELTE Egerv\'ary Research Group, Budapest}
\begin{document}
\maketitle

\begin{abstract}
An edge-colored graph is said to contain a rainbow-$F$ if it contains $F$ as a subgraph and every edge of $F$ is a distinct color. The problem of maximizing edges among $n$-vertex properly edge-colored graphs not containing a rainbow-$F$, known as the rainbow Tur\'an problem, was initiated by Keevash, Mubayi, Sudakov and Verstra\"ete.  We investigate a variation of this problem with the additional restriction that the graph is planar, and we denote the corresponding extremal number by $\ex_{\p}^*(n,F)$.  In particular, we determine $\ex_{\p}^*(n,P_5)$, where $P_5$ denotes the $5$-vertex path.   
\end{abstract}

\section{Introduction}
An edge-coloring of a graph $F$ is a \emph{proper edge-coloring} if no two incident edges of $F$ receive the same color. 
An edge-colored graph $F$ is called \emph{rainbow} if no two edges of $F$ receive the same color, and in short we call it a \emph{rainbow-$F$}.
A properly edge-colored graph $G$ is called \emph{rainbow-$F$-free} if it contains no rainbow-$F$ as a subgraph with the underlying edge-coloring of $G$. 

\begin{definition}
Let $n$ be a positive integer and $F$ be a fixed graph. The rainbow Tur\'an number of $F$, denoted as $\ex^{*}(n,F)$, is the maximum number of edges in a properly edge-colored $n$-vertex rainbow-$F$-free graph. 
\end{definition}
Clearly $\ex(n,F)\leq \ex^{*}(n,F)$. Indeed, any properly edge-colored $F$-free graph is rainbow-$F$-free. In fact, Keevash, Mubayi, Sudakov, and Verstra\"ete~\cite{8} proved the following.  
\begin{theorem}[Keevash et al.~\cite{8}]
For any graph $F$ and positive integer $n$,
\[\ex(n,F)\leq\ex^{*}(n,F)\leq \ex(n,F)+o(n^2).\]
\end{theorem}
Let $P_k$ denote a path on $k$ vertices.
Recently Johnston, Palmer and Sarkar~\cite{6} initiated the study of rainbow Tur\'an number of paths. 
They obtained sharp upper bounds for the rainbow Tur\'an number of $P_4$ and $P_5$. 
\begin{theorem}[Johnston, Palmer and Sarkar~\cite{6}] \label{sx1} $\empty$
\begin{enumerate}
\item If $n$ is divisible by $4$, then $\ex^*(n,P_4)={3n}/{2}.$
\item If $n$ is divisible by $8$, then $\ex^*(n,P_5)=2n.$  
\end{enumerate}
\end{theorem}
In the same paper they obtained a general upper bound which works for all $P_k$. Moreover they obtained constructions giving lower bound and posed their conjecture.
\begin{theorem}[Johnston, Palmer and Sarkar~\cite{6}] For any positive integers $n$ and $k$, $$\left(\frac{k-1}{2}\right)n+O(1)\leq\ex^*(n,P_k)\leq \left\lceil{\frac{3k-2}{2}}\right\rceil n.$$    
\end{theorem}
\begin{conjecture}[Johnston, Palmer and Sarkar~\cite{6}]
For $k\geq 6$, $$\ex^*(n,P_k)= \left(\frac{k-1}{2}\right)n+O(1).$$   
\end{conjecture}
Recently Ergemlidze, Gy\H{o}ri and Methuku~\cite{1} improved the upper bound.
\begin{theorem}[Ergemlidze, Gy\H{o}ri and Methuku~\cite{1}] For any positive integers $n$ and $k$, $$\ex^*(n,P_k)\leq \left(\frac{9k+5}{7}\right)n.$$ \end{theorem}

Continuing the study, Halfpap~\cite{3} obtained a sharp upper bound for the rainbow Tur\'an number of $P_6$ and the result partially confirms the conjecture of Johnston, Palmer and Sarkar.
\begin{theorem}[Halfpap~\cite{3}] For any positive integer $n$,  $\ex^*(n,P_6)\leq {5n}/{2}.$
\end{theorem}
The articles~\cite{2,4,5,7} contain  further results on rainbow-type extremal graph problems. In this paper, we consider rainbow Tur\'an numbers in the family of planar graphs. The following notations are important in the upcoming sections.

Let $G$ be a graph. We denote the vertex and the edge sets of $G$ by $V(G)$ and $E(G)$, respectively. The number of vertices and the number of edges of $G$ are denoted by $v(G)$ and $e(G)$, respectively.  For a vertex $v\in V(G)$, $d_G(v)$ denotes the degree of $v$. We may omit the subscript whenever there is no ambiguity about the underlying graph. 

For two vertices $u, v\in V(G)$, $\mathrm{dist}_G(u,v)$ denotes the distance between $u$ and $v$. Let $N_1(v)$ and $N_2(v)$ denote the set of neighbors and that of the second neighbors of $v$, respectively. i.e., $N_1(v)=\{u\ | \ u\in V(G)\ \text{and} \ \mathrm{dist}_G(u,v)=1\}$ and $N_2(v)=\{u \ | \ u\in V(G)\ \text{and} \ \mathrm{dist}_G(u,v)=2\}$. Denote $N[v]=N(v)\cup \{v\}$. For simplicity, we may use the notation $N(v)$ instead of $N_1(v)$. 

The path and the cycle on $k$ vertices are denoted by $P_k$ and $C_k$, respectively. A cycle on $k$ vertices is also referred as a $k$-cycle. For any $S\subseteq V(G)$, $G[S]$ denotes the subgraph of $G$ induced by $S$.

For graphs $G$ and $H$, we denote by $G \cup H$ the disjoint union of $G$ and $H$, and we denote by $G + H$ the join of $G$ and $H$, i.e.\ the graph obtained by connecting each pair of vertices between a vertex-disjoint copy of $G$ and $H$.

\section{Rainbow planar Tur\'an number of paths}
\begin{definition}
Let $n$ be a positive integer and $F$ be a fixed graph. The rainbow planar Tur\'an number of $F$, denoted by $\ex_{\p}^{*}(n,F)$, is the maximum number of edges in an $n$-vertex rainbow-$F$-free edge-colored planar graph. 
\end{definition}
A properly edge-colored planar graph is rainbow-$P_3$-free if and only if the graph is a matching plus isolated vertices. 
Thus, $\ex_{\p}^{*}(n,P_3)=\left\lfloor{{n}/{2}}\right\rfloor$.

An extremal construction showing the first bound in Theorem~\ref{sx1} is sharp can be obtained by taking  disjoint copies of $K_4$ with a proper $3$-edge coloring as pictured in Figure~\ref{f2}.
Clearly, the graph is planar, and hence we have $\ex_{\mathcal{P}}^*(n,P_4)\leq {3n}/{2}$. 


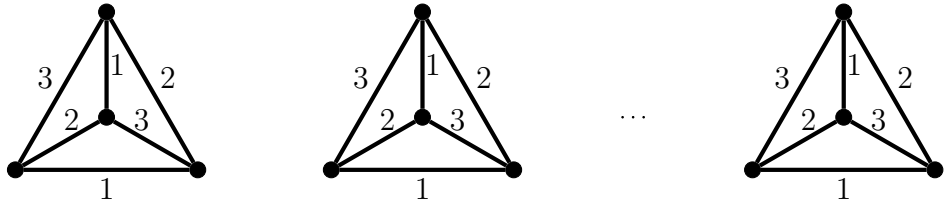
\begin{figure}[ht]
 \begin{center}
 \begin{tikzpicture}[scale=0.7]
  \tikzstyle{vertex}=[draw,circle,fill=black,minimum size=6,inner sep=0]
  
  \node[vertex] (a) at (0:0) {};
  \node[vertex] (b) at (90:2) {};
  \node[vertex] (c) at (210:2) {};
  \node[vertex] (d) at (330:2) {};
  
  \draw[ultra thick] (a) -- (b) node[midway, xshift=4pt] {1};
  \draw[ultra thick] (a) -- (c) node[midway, xshift=4pt, yshift=9pt] {2};
  \draw[ultra thick] (a) -- (d) node[midway, xshift=-4pt, yshift=9pt] {3};
  \draw[ultra thick] (b) -- (c) node[midway, xshift=-6pt, yshift=5pt] {3};
  \draw[ultra thick] (b) -- (d) node[midway, xshift=6pt, yshift=5pt] {2};
  \draw[ultra thick] (c) -- (d) node[midway, yshift=-7pt] {1};
  
  \begin{scope}[shift={(6,0)}]
  \node[vertex] (a) at (0:0) {};
  \node[vertex] (b) at (90:2) {};
  \node[vertex] (c) at (210:2) {};
  \node[vertex] (d) at (330:2) {};
  
  \draw[ultra thick] (a) -- (b) node[midway, xshift=4pt] {1};
  \draw[ultra thick] (a) -- (c) node[midway, xshift=4pt, yshift=9pt] {2};
  \draw[ultra thick] (a) -- (d) node[midway, xshift=-4pt, yshift=9pt] {3};
  \draw[ultra thick] (b) -- (c) node[midway, xshift=-6pt, yshift=5pt] {3};
  \draw[ultra thick] (b) -- (d) node[midway, xshift=6pt, yshift=5pt] {2};
  \draw[ultra thick] (c) -- (d) node[midway, yshift=-7pt] {1};
  \end{scope}
  
  \draw[fill] (9.8,0) circle (0.4pt);
  \draw[fill] (10,0) circle (0.4pt);
  \draw[fill] (10.2,0) circle (0.4pt);
  
  \begin{scope}[shift={(14,0)}]
  \node[vertex] (a) at (0:0) {};
  \node[vertex] (b) at (90:2) {};
  \node[vertex] (c) at (210:2) {};
  \node[vertex] (d) at (330:2) {};
  
  \draw[ultra thick] (a) -- (b) node[midway, xshift=4pt] {1};
  \draw[ultra thick] (a) -- (c) node[midway, xshift=4pt, yshift=9pt] {2};
  \draw[ultra thick] (a) -- (d) node[midway, xshift=-4pt, yshift=9pt] {3};
  \draw[ultra thick] (b) -- (c) node[midway, xshift=-6pt, yshift=5pt] {3};
  \draw[ultra thick] (b) -- (d) node[midway, xshift=6pt, yshift=5pt] {2};
  \draw[ultra thick] (c) -- (d) node[midway, yshift=-7pt] {1};
  \end{scope}
 \end{tikzpicture}
 \caption{A $3$-edge colored graph which is rainbow-$P_4$-free.}
 \label{f2}
 \end{center}
\end{figure}

As mentioned earlier, this paper is mainly focused on determining the rainbow planar Tur\'an number of $P_5$, i.e., $\ex^*_{\p}(n,P_5)$. Interestingly, the values are the same as the sharp upper bound of the rainbow planar Tur\'an number of $P_4$ which is stated in Theorem~\ref{sx1}. 

For longer paths, the following theorem verifies that when $k\ge 8$ we have maximal planar graphs avoiding rainbow paths of length $k$. 
However, the case of $P_6$ and $P_7$ still seems quite interesting.
In section 9, we pose our conjectures concerning the rainbow planar Tur\'an number of these paths. 
\begin{theorem}~\label{y1}
For $n\geq k\geq 8$, 
\[\ex_{\p}^{*}(n,P_k)=3n-6.\]
\end{theorem}
\begin{proof}
First, let $n$ be even. Take a double-wheel, $W_n=(K_1\cup K_1)+C_{n-2}$. Note that $W_n$ has two vertices of degree $n-2$ and all the remaining vertices are of degree $4$. Let the $(n-2)$-cycle be $v_1v_2v_3\dots v_{n-1}v_{n-2}v_1$ and the two high-degree vertices be $u$ and $w$. 

We define an edge-coloring $c$ of $W_n$, which uses $2n-2$ distinct colors, as follows. Let $c(uv_i)=a_i$ and $c(wv_i)=b_i$ for any $i\in\{1,2,\dots,n-2\}$, and alternately color the edges of the $(n-2)$-cycle using two other colors $a$ and $b$.

Now we show that with this edge-coloring, $W_n$ is rainbow-$P_k$-free for any $k\geq 8$. Consider a path of length $7$ in $W_n$. This path contains at most $4$ edges which are not in the $(n-2)$-cycle. Thus, the number of edges which are on this cycle is at least $3$ and hence the path is not rainbow. Since $W_n$ is a maximal planar graph, we have $\ex_{\p}^{*}(n,P_k)=3n-6$.

Now let $n$ be odd. In this case, take a maximal planar graph $H_n=K_2+P_{n-2}$. Let the path $P_{n-2}$ be $v_1v_2v_3\dots v_{n-1}v_{n-2}$ and let the remaining two vertices be $u$ and $w$. Define an edge-coloring $c$ of $H_n$, which uses $2n-3$ distinct colors, as follows. Let $c(uv_i)=a_i$ for any $i\in\{2,3,\dots, n-3\}$, let $c(wv_i)=b_i$ for any $i\in\{1,2,3\dots,n-2\}$, let $c(uw)=d$, and alternately color the edges of the even cycle $uv_1v_2v_3\dots v_{n-2}u$ using two other colors $a$ and $b$. 

Again, it is not difficult to see that $c$ is a proper edge-coloring and any path of length $7$ contains at least three edges of the $(n-1)$-cycle $uv_1v_2\dots v_{n-2}u$. Thus, every $P_8$ of $H_n$ is not rainbow. Since $H_n$ is a maximal planar graph, we have $\ex_{\p}^*(n,P_k)=3n-6$. 
\end{proof}

\section{Main results and extremal constructions}
The following theorem is our main result and it gives the exact value of the rainbow planar Tur\'an number of $P_5$. 
\begin{theorem}\label{t4}
For $n\geq 4$, $\ex_{\p}^{*}(n,P_5)=\left\lfloor{ {3n}/{2}}\right\rfloor.$ 
\end{theorem}
The extremal constructions pictured in Figure~\ref{f2} contain no rainbow-$P_5$, and thereby verify that the bound $\ex_{\mathcal{P}}^*(n,P_5)\leq \left\lfloor{{3n}/{2}}\right\rfloor$ is sharp. Next we give a properly edge-colored rainbow-$P_5$-free planar graph $G_n$ on $n\geq5$ vertices with  $e(G_n)=\left\lfloor{{3n}/{2}}\right\rfloor.$
\subsubsection*{Constructions of $G_5$ and $G_7$.}
A $5$-vertex and $7$-edge planar graph $G_5$ with a proper edge-coloring is shown in Figure~\ref{f3}~(left). Moreover, a $7$-vertex and $10$-edge planar graph $G_7$ with a proper edge-coloring is shown in Figure~\ref{f3}~(right). Observe that both $G_5$ and $G_7$ are rainbow-$P_5$-free graphs and $e(G_5)=\left\lfloor{{3n}/{2}}\right\rfloor=7$ and $e(G_7)=\left\lfloor{{3n}/{2}}\right\rfloor=10$.


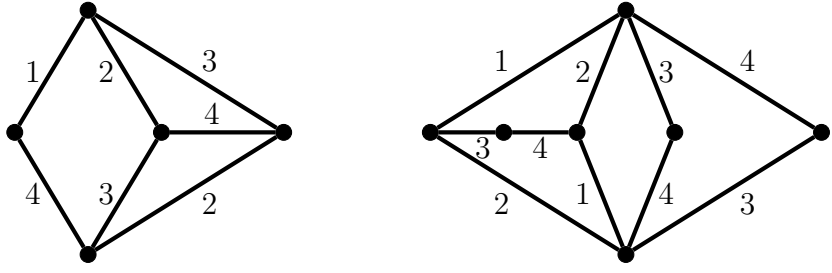
\begin{figure}[ht]
\begin{center}
 \begin{tikzpicture}[scale=0.65]
  \tikzstyle{vertex}=[draw,circle,fill=black,minimum size=6,inner sep=0]
  
  \node[vertex] (a1) at (0,2.5) {};
  \node[vertex] (a2) at (-1.5,0) {};
  \node[vertex] (a3) at (1.5,0) {};
  \node[vertex] (a4) at (4,0) {};
  \node[vertex] (a5) at (0,-2.5) {};
  
  \draw[ultra thick] (a1) -- (a2) node[midway, xshift=-7pt, yshift=0pt] {1};
  \draw[ultra thick] (a1) -- (a3) node[midway, xshift=-7pt] {2};
  \draw[ultra thick] (a1) -- (a4) node[midway, xshift=9pt, yshift=4pt] {3};
  \draw[ultra thick] (a3) -- (a4) node[midway, xshift=-4pt, yshift=7pt] {4};
  \draw[ultra thick] (a5) -- (a2) node[midway, xshift=-7pt, yshift=0pt] {4};
  \draw[ultra thick] (a5) -- (a3) node[midway, xshift=-7pt] {3};
  \draw[ultra thick] (a5) -- (a4) node[midway, xshift=9pt, yshift=-4pt] {2};
  
  \begin{scope}[shift={(11,0)}]
  \node[vertex] (a1) at (0,2.5) {};
  \node[vertex] (a2) at (-4,0) {};
  \node[vertex] (a3) at (-2.5,0) {};
  \node[vertex] (a4) at (-1,0) {};
  \node[vertex] (a5) at (1,0) {};
  \node[vertex] (a6) at (4,0) {};
  \node[vertex] (a7) at (0,-2.5) {};
  
  \draw[ultra thick] (a1) -- (a2) node[midway, xshift=-10pt, yshift=4pt] {1};
  \draw[ultra thick] (a1) -- (a4) node[midway, xshift=-7pt, yshift=0pt] {2};
  \draw[ultra thick] (a1) -- (a5) node[midway, xshift=6pt] {3};
  \draw[ultra thick] (a1) -- (a6) node[midway, xshift=9pt, yshift=4pt] {4};
  \draw[ultra thick] (a2) -- (a3) node[midway, xshift=6pt, yshift=-6pt] {3};
  \draw[ultra thick] (a3) -- (a4) node[midway, yshift=-6pt] {4};
  \draw[ultra thick] (a7) -- (a2) node[midway, xshift=-10pt, yshift=-4pt] {2};
  \draw[ultra thick] (a7) -- (a4) node[midway, xshift=-7pt, yshift=0pt] {1};
  \draw[ultra thick] (a7) -- (a5) node[midway, xshift=6pt] {4};
  \draw[ultra thick] (a7) -- (a6) node[midway, xshift=9pt, yshift=-4pt] {3};
  \end{scope}
 \end{tikzpicture}
 \end{center}
\caption{The graphs $G_5$ and $G_7$.}
\label{f3}
\end{figure}

\subsubsection*{Constructions of $G_n$ when $n$ is even and $n\geq 4$.}
For $n=4$, we may take $G_n$ to be $K_4$. Its proper $3$-edge coloring is already given in Figure~\ref{f2}. Figure~\ref{F1} shows the constructions of $G_n$ with a proper $3$-edge coloring when $n\geq 6$ is even. Clearly, $G_n$ is rainbow-$P_5$-free. Moreover, $G_n$ is $3$-regular and $e(G_n)=\left\lfloor{{3n}/{2}}\right\rfloor$.

\subsubsection*{Constructions of $G_n$ when $n$ is odd and $n\geq 9$.}
For any odd number $n\geq 9$, we may take $G_n$ to be the disjoint union of $G_{n-5}$ and $G_5$, i.e., $G_n=G_{n-5}\cup G_5$. Then $e(G_n)=\left\lfloor{{3(n-5)/}{2}}\right\rfloor+7=\left\lfloor{{3n}/{2}}\right\rfloor$. By induction, $G_n$ is rainbow-$P_5$-free.
\begin{figure}[ht]
 \begin{center}
 \begin{tikzpicture}[scale=0.9]
  \tikzstyle{vertex}=[draw,circle,fill=black,minimum size=6,inner sep=0]
  
  \node[vertex] (a1) at (1,1) [label={[xshift=-7pt, yshift=-1pt] $a_1$}] {};
  \node[vertex] (a2) at (2,1) {};
  \node[vertex] (a3) at (3,1) {};
  \node[vertex] (a4) at (4,1) {};
  \node[vertex] (a5) at (5.5,1) {};
  \node[vertex] (a6) at (6.5,1) {};
  \node[vertex] (a7) at (7.5,1) [label={[xshift=11pt, yshift=-4pt] $a_{n/2}$}] {};
  \node[vertex] (b1) at (1,0) [label={[xshift=-6pt, yshift=-22pt] $b_1$}] {};
  \node[vertex] (b2) at (2,0) {};
  \node[vertex] (b3) at (3,0) {};
  \node[vertex] (b4) at (4,0) {};
  \node[vertex] (b5) at (5.5,0) {};
  \node[vertex] (b6) at (6.5,0) {};
  \node[vertex] (b7) at (7.5,0) [label={[xshift=12pt, yshift=-26pt] $b_{n/2}$}] {};
  
  \draw[fill] (4.6,1) circle (0.4pt);
  \draw[fill] (4.75,1) circle (0.4pt);
  \draw[fill] (4.9,1) circle (0.4pt);
  
  \draw[fill] (4.6,0) circle (0.4pt);
  \draw[fill] (4.75,0) circle (0.4pt);
  \draw[fill] (4.9,0) circle (0.4pt);
  
  \draw[ultra thick] (a1) -- (a2) node[midway, yshift=7pt] {1};
  \draw[ultra thick] (a2) -- (a3) node[midway, yshift=7pt] {2};
  \draw[ultra thick] (a3) -- (a4) node[midway, yshift=7pt] {1};
  \draw[ultra thick] (a5) -- (a6) node[midway, yshift=7pt] {2};
  \draw[ultra thick] (a6) -- (a7) node[midway, yshift=7pt] {1};
  \draw[ultra thick] (b1) -- (b2) node[midway, yshift=-7pt] {1};
  \draw[ultra thick] (b2) -- (b3) node[midway, yshift=-7pt] {2};
  \draw[ultra thick] (b3) -- (b4) node[midway, yshift=-7pt] {1};
  \draw[ultra thick] (b5) -- (b6) node[midway, yshift=-7pt] {2};
  \draw[ultra thick] (b6) -- (b7) node[midway, yshift=-7pt] {1};
  
  \draw[ultra thick] (a1)  to [bend left=60,looseness=0.75] node[midway, yshift=7pt] {2} (a7);
  \draw[ultra thick] (b1)  to [bend right=60,looseness=0.75] node[midway, yshift=-7pt] {2} (b7);
  
  \draw[ultra thick] (a1) -- (b1) node[midway, xshift=6pt] {3};
  \draw[ultra thick] (a2) -- (b2) node[midway, xshift=6pt] {3};
  \draw[ultra thick] (a3) -- (b3) node[midway, xshift=6pt] {3};
  \draw[ultra thick] (a4) -- (b4) node[midway, xshift=6pt] {3};
  \draw[ultra thick] (a5) -- (b5) node[midway, xshift=6pt] {3};
  \draw[ultra thick] (a6) -- (b6) node[midway, xshift=6pt] {3};
  \draw[ultra thick] (a7) -- (b7) node[midway, xshift=6pt] {3};
  
  \begin{scope}[shift={(9,0)}]
  \node[vertex] (a1) at (1,1) [label={[xshift=-7pt, yshift=-1pt] $a_1$}] {};
  \node[vertex] (a2) at (2,1) {};
  \node[vertex] (a3) at (3,1) {};
  \node[vertex] (a4) at (4,1) {};
  \node[vertex] (a5) at (5.5,1) {};
  \node[vertex] (a6) at (6.5,1) {};
  \node[vertex] (a7) at (7.5,1) {};
  \node[vertex] (a8) at (8.5,1) [label={[xshift=11pt, yshift=-4pt] $a_{n/2}$}] {};
  \node[vertex] (b1) at (1,0) [label={[xshift=-6pt, yshift=-22pt] $b_1$}] {};
  \node[vertex] (b2) at (2,0) {};
  \node[vertex] (b3) at (3,0) {};
  \node[vertex] (b4) at (4,0) {};
  \node[vertex] (b5) at (5.5,0) {};
  \node[vertex] (b6) at (6.5,0) {};
  \node[vertex] (b7) at (7.5,0) {};
  \node[vertex] (b8) at (8.5,0) [label={[xshift=12pt, yshift=-26pt] $b_{n/2}$}] {};
  
  \draw[fill] (4.6,1) circle (0.4pt);
  \draw[fill] (4.75,1) circle (0.4pt);
  \draw[fill] (4.9,1) circle (0.4pt);
  
  \draw[fill] (4.6,0) circle (0.4pt);
  \draw[fill] (4.75,0) circle (0.4pt);
  \draw[fill] (4.9,0) circle (0.4pt);
  
  \draw[ultra thick] (a1) -- (a2) node[midway, yshift=7pt] {1};
  \draw[ultra thick] (a2) -- (a3) node[midway, yshift=7pt] {2};
  \draw[ultra thick] (a3) -- (a4) node[midway, yshift=7pt] {1};
  \draw[ultra thick] (a5) -- (a6) node[midway, yshift=7pt] {2};
  \draw[ultra thick] (a6) -- (a7) node[midway, yshift=7pt] {1};
  \draw[ultra thick] (a7) -- (a8) node[midway, yshift=7pt] {2};
  \draw[ultra thick] (b1) -- (b2) node[midway, yshift=-7pt] {1};
  \draw[ultra thick] (b2) -- (b3) node[midway, yshift=-7pt] {2};
  \draw[ultra thick] (b3) -- (b4) node[midway, yshift=-7pt] {1};
  \draw[ultra thick] (b5) -- (b6) node[midway, yshift=-7pt] {2};
  \draw[ultra thick] (b6) -- (b7) node[midway, yshift=-7pt] {1};
  \draw[ultra thick] (b7) -- (b8) node[midway, yshift=-7pt] {2};
  
  \draw[ultra thick] (a1)  to [bend left=60,looseness=0.65] node[midway, yshift=7pt] {3} (a8);
  \draw[ultra thick] (b1)  to [bend right=60,looseness=0.65] node[midway, yshift=-7pt] {3} (b8);
  
  \draw[ultra thick] (a1) -- (b1) node[midway, xshift=6pt] {2};
  \draw[ultra thick] (a2) -- (b2) node[midway, xshift=6pt] {3};
  \draw[ultra thick] (a3) -- (b3) node[midway, xshift=6pt] {3};
  \draw[ultra thick] (a4) -- (b4) node[midway, xshift=6pt] {3};
  \draw[ultra thick] (a5) -- (b5) node[midway, xshift=6pt] {3};
  \draw[ultra thick] (a6) -- (b6) node[midway, xshift=6pt] {3};
  \draw[ultra thick] (a7) -- (b7) node[midway, xshift=6pt] {3};
  \draw[ultra thick] (a8) -- (b8) node[midway, xshift=6pt] {1};
  \end{scope}
 \end{tikzpicture}
 \end{center}
 \label{F1}
\caption{A $3$-edge coloring of $G_n$ when $n/2$ is even (left) and when $n/2$ is odd (right) for even integers $n \ge 6$.}
\end{figure}
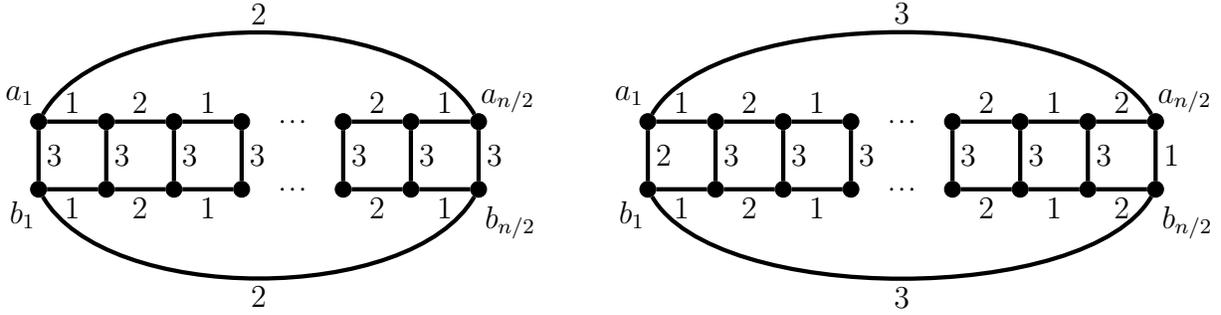

\section{Outline of the proof of Theorem~\ref{t4}}
From the extremal constructions discussed in the previous section, we have $\ex_{\p}^*(n,P_5)\geq \left\lfloor{{3n}/{2}}\right\rfloor$. In the upcoming sections, we prove $\ex_{\p}^*(n,P_5)\leq\left\lfloor{{3n}/{2}}\right\rfloor$, thereby completing the proof of Theorem~\ref{t4}. 
The proof includes an induction argument and discharging methods.
To manage the flow of the proof, we break it into four sections.  

In Section~\ref{12}, some basic preliminaries for the proof are given. The section examines some important graph structures together with proper edge-coloring schemes under the assumption that the graph avoids a rainbow-$P_5$.  

In Section~\ref{er}, we  analyze thoroughly the situations for which we can finish the proof by induction.  
In this section, we also show why the planar Tur\'an number and the rainbow planar Tur\'an number of $P_5$ behave quite differently. 

Section~\ref{fg} gives more refinements on the graph by avoiding vertices of degree at least five and two degree-$4$ vertices sharing a common vertex. Moreover, in this section we show each degree-$4$ vertex has at least two degree-$2$ vertices in its neighbor or second neighbor.  

We finish the proof in Section~\ref{cd}.
In this section we define a modification $\Tilde{d}$ of the degree function, for which $\Tilde{d}(v)\leq 3$ holds for each vertex $v$ of the graph. 
This lets us to conclude the proof using the handshaking lemma.    

\section{Preliminaries}\label{12}
We need the following basic preliminary proper edge-coloring schemes which avoid a rainbow-$P_5$. 
The results are important in the upcoming sections of the proof. 

\begin{definition}
The \emph{bow tie graph} is two $3$-cycles sharing exactly one vertex (see Figure~\ref{ca1} (left)).
We refer to a bow tie with its vertices labelled as in the figure as a $(u_1u_2u, uu_3u_4)$-bow tie. 
\end{definition}

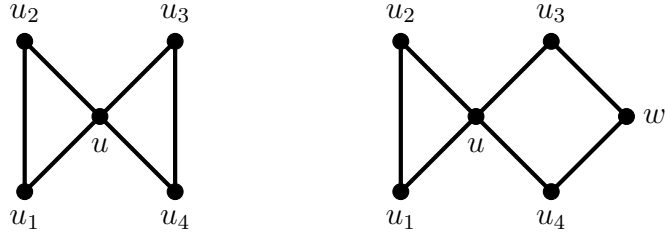
\begin{figure}[ht]
\centering
\begin{tikzpicture}[scale=0.25]
\draw[ultra thick](0,0)--(-4,-4)--(-4,4)--(0,0)--(4,4)--(4,-4)--(0,0);
\draw[fill=black](-4,-4)circle(12pt);
\draw[fill=black](-4,4)circle(12pt);
\draw[fill=black](4,4)circle(12pt);
\draw[fill=black](4,-4)circle(12pt);
\draw[fill=black](0,0)circle(12pt);
\node at (-4,-5.5) {$u_1$};
\node at (-4,5.5) {$u_2$};
\node at (4,5.5) {$u_3$};
\node at (4,-5.5) {$u_4$};
\node at (0,-1.5) {$u$};

\begin{scope}[shift={(20,0)}]
\draw[ultra thick](0,0)--(-4,-4)--(-4,4)--(0,0)--(4,4)(4,-4)--(0,0)(4,4)--(8,0)--(4,-4);
\draw[fill=black](-4,-4)circle(12pt);
\draw[fill=black](-4,4)circle(12pt);
\draw[fill=black](4,4)circle(12pt);
\draw[fill=black](4,-4)circle(12pt);
\draw[fill=black](0,0)circle(12pt);
\draw[fill=black](8,0)circle(12pt);
\node at (-4,-5.5) {$u_1$};
\node at (-4,5.5) {$u_2$};
\node at (4,5.5) {$u_3$};
\node at (4,-5.5) {$u_4$};
\node at (0,-1.5) {$u$};
\node at (9.5,0) {$w$};
\end{scope}
\end{tikzpicture}\qquad
\label{ca1}
\caption{The bow tie graph and the fish graph.}
\end{figure}

\begin{lemma}[The bow tie coloring schemes]\label{c2}
Consider a $(u_1u_2u,uu_3u_4)$-bow tie with a proper edge-coloring $c$ which avoids a rainbow-$P_5$. Without loss of generality, we can assume $c(uu_i)=i$ for any $i\in\{1,2,3,4\}$. Then $c(u_1u_2)=c(u_3u_4)$ and this color is not in $\{1,2,3,4\}$.
\end{lemma}
\begin{proof}
Consider the $5$-vertex path $u_1u_2uu_3u_4$. Since $c$ is a rainbow-$P_5$-free edge-coloring, it can be seen that either $c(u_1u_2)=c(u_3u_4)$ and this color is not in $\{1,2,3,4\}$, or $c(u_1u_2)=3$, or $c(u_3u_4)=2$. 

Suppose the first case does not hold. Without loss of generality, we can assume $c(u_1u_2)=3$. Considering the paths $u_3u_4uu_1u_2$ and $u_3u_4uu_2u_1$, we obtain that $c(u_3u_4)$ is respectively in $\{1,3\}$ and $\{2,3\}$. This implies, $c(u_3u_4)=3$. This is a contradiction as $uu_3$ is an edge incident to $u_3u_4$ and yet colored with $3$. This completes the proof of Lemma~\ref{c2}.
\end{proof}

\begin{definition}
The \emph{fish graph} is a $3$-cycle and a $4$-cycle sharing exactly one vertex (see Figure~\ref{ca1}~(right)).
We refer to the fish graph with vertices labelled as in the figure as a $(u_1u_2u, uu_3wu_4)$-fish. 
\end{definition}
\begin{lemma}[The fish graph coloring schemes]\label{c4}
Given a $(u_1u_2u,uu_3wu_4)$-fish with a proper edge-coloring $c$ which avoids a rainbow-$P_5$, we have $c(wu_3)=c(uu_4)$ and $c(wu_4)=c(uu_3)$, moreover either $c(u_1u_2)=c(uu_3)$ or $c(u_1u_2)=c(uu_4)$. 
\end{lemma}
\begin{proof}
Without loss of generality, we can assume $c(uu_i)=i$ for any $i\in\{1,2,3,4\}$. Clearly, $c(u_1u_2)$ is different from $1$ and $2$. Suppose for contradiction $c(u_1u_2)\notin\{3,4\}$. The we can assume $c(u_1u_2)=5$. Considering the paths $wu_3uu_2u_1$ and $wu_3uu_1u_2$, we obtain that $c(u_3w)$ is respectively in $\{2,5\}$ and $\{1,5\}$, which implies $c(u_3w)=5$. Consequently, considering the path $u_3wu_4uu_1$, we get $c(wu_4)=1$. But this results in a contradiction as we get a rainbow-$P_5$ path $u_3wu_4uu_2$, whose edges are of colors 5, 1, 4, and 2, in this order. 

Therefore, by the symmetry of the fish graph, we can assume $c(u_1u_2)=3$. Considering the paths $wu_4uu_1u_2$ and $wu_4uu_2u_1$, we get $c(wu_4)$ is in $\{1,3\}$ and $\{2,3\}$, respectively. This implies, $c(wu_4)=3$. For a similar reason, we obtain $c(wu_3)=4$. This completes the proof of Lemma~\ref{c4}. 
\end{proof}

\begin{definition}
A \emph{medium-pair} is three distinct cherries with common end vertices, i.e., $K_{2,3}$, see Figure~\ref{vg1}~(left). A \emph{heavy-pair} is four distinct cherries with common end vertices, i.e., $K_{2,4}$, see Figure~\ref{vg1}~(right). 
\end{definition}

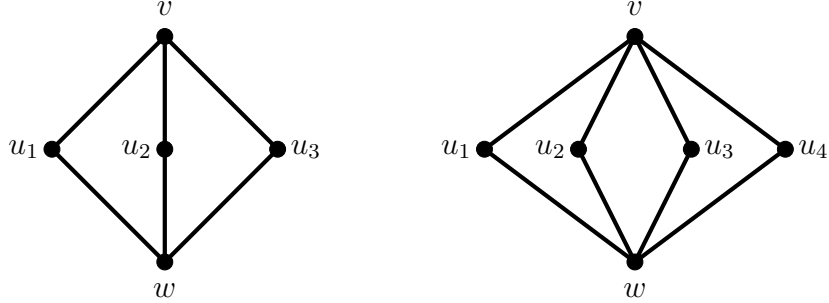
\begin{figure}[ht]
\centering
\begin{tikzpicture}[scale=0.25]
\draw[ultra thick](0,-6)--(0,0)(0,-6)--(6,0)(0,-6)--(-6,0)(0,6)--(6,0)(0,6)--(-6,0)(0,6)--(0,0);
\draw[fill=black] (0,-6) circle (12pt);
\draw[fill=black] (0,6) circle (12pt);
\draw[fill=black] (0,0) circle (12pt);
\draw[fill=black] (6,0) circle (12pt);
\draw[fill=black] (-6,0) circle (12pt);
\node at (0,-7.5) {$w$};
\node at (0,7.5) {$v$};
\node at (-7.5,0) {$u_1$};
\node at (-1.5,0) {$u_2$};
\node at (7.5,0) {$u_3$};

\begin{scope}[shift={(25,0)}]
\draw[ultra thick](0,-6)--(3,0)(0,-6)--(8,0)(0,-6)--(-3,0)(0,-6)--(-8,0)(0,6)--(3,0)(0,6)--(8,0)(0,6)--(-3,0)(0,6)--(-8,0);
\draw[fill=black] (0,-6) circle (12pt);
\draw[fill=black] (0,6) circle (12pt);
\draw[fill=black] (3,0) circle (12pt);
\draw[fill=black] (8,0) circle (12pt);
\draw[fill=black] (-3,0) circle (12pt);
\draw[fill=black] (-8,0) circle (12pt);
\node at (0,-7.5) {$w$};
\node at (0,7.5) {$v$};
\node at (-9.5,0) {$u_1$};
\node at (-4.5,0) {$u_2$};
\node at (9.5,0) {$u_4$};
\node at (4.5,0) {$u_3$};
\end{scope}
\end{tikzpicture}
\label{vg1}
\caption{A medium-pair and a heavy-pair.}
\end{figure}

\begin{lemma}[The medium-pair coloring schemes]\label{l5}
Given a medium-pair with cherries $vu_iw$, where $i\in\{1,2,3\}$, let $c$ be a proper edge-coloring of it which avoids a rainbow-$P_5$. Then $c$ uses either 
\begin{enumerate}
\item four colors, where two of which are applied in edges of two of its cherries, or 
\item three colors, where $c(vu_i)\in\{c(wu_1),\ c(wu_2),\ c(wu_3)\}$ for any $i\in\{1,2,3\}$.  
\end{enumerate}
\end{lemma}
For instance in the first case, we may have the cherries $vu_1w$ and $vu_2w$ to which $c$ assigns only two colors: $c(vu_1)=c(wu_2)$ and $c(vu_2)=c(wu_1)$. 
On the other hand, in the second case, we may have the edge-coloring $c(vu_1)=c(wu_3),  \ c(vu_2)=c(wu_1)$ and $c(vu_3)=c(wu_2)$.  

For the purpose of simplicity, we sometimes refer to the first case as a \emph{four-coloring scheme} and the second case as a \emph{three-coloring scheme} for the medium-pair. 
Next we give the proof of Lemma~\ref{l5}.

\begin{proof}[Proof of Lemma~\ref{l5}]
Without loss of generality, we can assume $c(vu_i)=i$ for $i\in\{1,2,3\}$. Since $c$ is a rainbow-$P_5$-free coloring, it is not difficult to check that at most one edge in $\{wu_1,wu_2,wu_3\}$ receives a color different from 1, 2, 3.

Suppose (2) does not occur. Then t most one edge in $\{wu_1,wu_2,wu_3\}$ receives a color not in $\{1,2,3\}$.
Without loss of generality, we can assume $c(wu_3)=4$. Considering the paths $u_3wu_2vu_1$ and $u_3wu_1vu_2$, we get $c(wu_2)=1$ and $c(wu_1)=2$. This completes the proof of Lemma~\ref{l5}.  
\end{proof}
\begin{lemma}\label{ws}
Let $G$ be a rainbow-$P_5$-free graph containing a medium pair and let $vu_iw$ be the $3$ cherries forming this medium-pair where $i\in\{1,2,3\}$. 
Then $v$ is adjacent to at most one vertex in $V(G)\setminus \{u_1,u_2, u_3,w\}$.
\end{lemma}
\begin{proof}
Let $S=\{u_1,u_2,u_3,w\}$.
We distinguish two cases based on the coloring scheme of the medium-pair as stated in Lemma~\ref{l5}.

\medskip

\noindent \textit{Case 1: The medium-pair is colored with the three-coloring scheme.}

\smallskip

Without loss of generality suppose, $c(vu_1)=c(wu_3)=1, \ c(vu_2)=c(wu_1)=2$ and $c(vu_3)=c(wu_2)=3$.
In this case, we show that there is no vertex $v^*\in V(G)\setminus S$ such that $vv^* \in E(G)$. Indeed, suppose such a vertex exists in $G$. Clearly, $c(vv^*)\notin \{1,2,3\}$, so we can assume $c(vv^*)=4$. 
However, this results in a contradiction as the $v^*vu_1wu_2$ is a rainbow-$P_5$.

\medskip

\noindent \textit{Case 2: The medium-pair is colored with the four-coloring scheme.}

\smallskip

Without loss of generality, we can assume $c(vu_1)=c(wu_2)=1,\  c(vu_2)=c(wu_1)=2,\ c(vu_3)=3$ and $c(wu_3)=4$. Assume that there exists a vertex $v^*\in V(G)\backslash S$ such that $vv^*\in E(G)$. Considering the paths $v^*vu_1wu_3$ and $v^*vu_2wu_3$, we obtain that $c(vv^*)$ is in $\{2,4\}$ and $\{1,4\}$, respectively. Thus, $c(u_1v)=4$. This implies that $v$ cannot be adjacent to more than a vertex in $V(G)\backslash S$. This completes the proof of Lemma~\ref{ws}.
\end{proof}

\begin{lemma}[The heavy-pair coloring schemes]\label{cv1}
Consider a heavy-pair using the cherries $\ vu_iw$, where $i\in\{1,2,3,4\}$, with a proper edge-coloring $c$ which avoids a rainbow-$P_5$. Then $c$ uses only four colors and without loss of generality, we can assume $c(wu_2)=c(vu_1)$,  $c(wu_1)=c(vu_2)$, $c(wu_4)=c(vu_3)$, and $c(wu_3)=c(vu_4)$.
\end{lemma}
\begin{proof}
Without loss of generality, we can assume $c(vu_i)=i$ for any $i\in\{1,2,3,4\}$. It is easy to show for each $i\in\{1,2,3,4\}$, we have $c(wu_i)\in\{1,2,3,4\}$. Indeed, suppose not and without loss of generality, assume $c(wu_1)=5$. Then considering the path $u_1wu_2vu_3$, we get $c(wu_2)=3$. But this results in a contradiction as we get rainbow-$P_5$ path $u_1wu_2vu_4$, whose edges are of colors 5, 3, 2, and 4, in this order.

Without loss of generality, we can assume $c(wu_1)=c(vu_2)=2$. Then considering the paths $u_2wu_1vu_3$ and $u_2wu_1vu_4$, we get $c(u_2w)=1$. Consequently, $c(u_3w)=4$ and $c(u_4w)=3$ hold, and this completes the proof Lemma~\ref{cv1}.
\end{proof}

\section{Breaking the problem}\label{er}
From now on, by a graph $G$ we mean an $n$-vertex planar graph with a proper edge-coloring $c$ which avoids a rainbow-$P_5$. The following are important observations for the proof. 

If $G$ contains two adjacent degree-2 vertices, then we can finish the proof by induction after deleting these two vertices. Moreover, if $G$ contains a vertex of degree at most 1, then again we can finish the proof by induction after deleting this vertex. Even more, we can also finish the proof by induction if we find a $k$-vertex component $H$ in $G$ such that $e(H)\leq {3k}/{2}$. We call such a component a \emph{nice-component}.

\begin{definition}
Let $H$ be a subgraph of $G$. A vertex $v$ in $H$ is called \emph{saturated} if $d_{H}(v)=d_{G}(v)$. 
\end{definition}

\begin{lemma}\label{c3}
Let $(u_1u_2u,uu_3u_4)$-bow tie be in $G$. Then $u_1, u_2, u_3$ and $u_4$ are saturated vertices. i.e., $d_G(u_i)=2$, $i\in\{1,2,3,4\}.$
\end{lemma}
\begin{proof} 
Denote $c(uu_i)=i$, where $i\in\{1,2,3,4\}$. From the bow tie coloring scheme, which is stated in Lemma~\ref{c2}, we may assume that $c(u_1u_2)=c(u_3u_4)=5$. Without loss of generality we verify that $u_2$ is a saturated vertex.

First we show that $u_2u_3\notin E(G)$, and with similar argument $u_2u_4\notin E(G)$. Indeed, suppose for contradiction $u_2u_3\in E(G).$ From the known colors of the $P_5$ path $u_4uu_3u_2u_1$, $c(u_3u_2)=4$. But this results a rainbow-$P_5$ path, namely $u_1uu_2u_3u_4$, with edge-colors from left to right $1,\ 2,\ 4$ and $5$ respectively.

Next suppose there is a vertex $w\in V(G)\backslash\{u,u_1,u_3,u_4\}$ such that $wu_2\in E(G)$. From the $P_5$ path $wu_2uu_3u_4$, $c(wu_2)\in\{3,5\}$. However, it can not be $5$, as $c(u_1u_2)=5$ and $u_1u_2$ is incident to the edge $wu_2$. This implies $c(wu_1)=3$, but this results a rainbow-$P_5$ $wu_2uu_4u_3$ with edge colors from left to right $3,\ 2,\ 4$ and $5$ respectively. This completes the proof of Lemma~\ref{c3}.
\end{proof}
From Lemma~\ref{c3} if $G$ contains a bow tie, then its degree-$2$ vertices are saturated and hence we have two adjacent degree-$2$ vertices to finish the proof by induction.  
\begin{lemma}\label{lm5}
If $G$ contains four $3$-cycles sharing an edge, then each vertex of the cycles is saturated.
\end{lemma}
\begin{proof}
Let the four $3$-cycles be $uwu_iw$, where $i\in\{1,2,3,4\}$. Clearly $G$ contains a heavy-pair $uu_iw$, where $i\in\{1,2,3,4\}$. 

From the heavy-pair coloring scheme given in Lemma~\ref{cv1}, assume $c(uu_1)=c(wu_2)=1$, $c(uu_2)=c(wu_1)=2$, $c(uu_3)=c(wu_4)=3$, $c(uu_4)=c(wu_3)=4$ and $c(uw)=5$. Denote $S=\{u,w,u_1,u_2,u_3,u_4\}$.

Observe that no two vertices in $\{u_1,u_2,u_3,u_4\}$ are adjacent. Indeed, suppose $u_1u_2\in E(G)$. Clearly $(u_1u_2u,uwu_3)$-bow tie is a subgraph in $G$. From Lemma~\ref{c3}, $u_1, u_2, u_3$ and $w$ are saturated vertices in $G$, which is not true and hence a contradiction. 

Let $x\in V(G)\backslash S$. It can be shown that $x$ is not adjacent to any vertex in $S$. To see this, 
suppose $xu\in E(G)$. From the $P_5$ paths $xuu_4wu_1,\ xuu_4wu_2$ and $xuu_2wu_3$, $c(xu)$ is respectively in $\{2,3\}$, $\{1,3\}$ and $\{1,4\}$, which is a contradiction. For a similar argument $xw\notin E(G)$.

Suppose $xu_1\in E(G).$ From the $P_5$ paths $xu_1wuu_3, xu_1wuu_4$ and $xu_1wu_3u$ $c(xu_1)$ is respectively in $\{3,5\}$, $\{4,5\}$ and $\{3,4\}$, which is again a contradiction. For similar reasons $xu_2, \ xu_3,\ xu_4\notin E(G)$. This completes the proof of Lemma~\ref{lm5}.
\end{proof}
From Lemma~\ref{lm5} observe that if $G$ contains four $3$-cycles sharing an edge and $S$ is the set of vertices of the cycles, then $G[S]$ is a $6$-vertex and $9$-edge nice-component, and we can finish the proof by induction. 
\begin{lemma}\label{tr}
If $G$ contains three $3$-cycles sharing an edge, then $G$ contains either a nice-component or a degree-$1$ vertex.
\end{lemma}
\begin{proof}
Let the three $3$-cycles be $uwu_iu$, where $i\in\{1,2,3\}$. The cherries $uu_iw$, where $i\in\{1,2,3\}$, form a medium-pair in $G$. Denote, $S=\{u,w,u_1,u_2,u_3\}$. We distinguish two cases considering the medium-pair coloring scheme stated in Lemma~\ref{l5}.

\medskip

\noindent \textit{Case 1: The medium-pair is colored with the four-coloring scheme.}

\smallskip

Without loss of generality assume $c(uw)=1$, $c(uu_1)=c(wu_2)=2$, $c(uu_2)=c(wu_1)=3$, $c(uu_3)=4$ and $c(wu_3)=5$. Let $x\in V(G)\backslash S$.  

It can be checked that each vertex in $\{u_1,u_2,u_3\}$ is a saturated vertex. For completeness, we show for $u_1$, but similar arguments can be given to show for the rest. It is easy to check that both $u_1u_2, u_1u_3\notin E(G)$. Suppose $xu_1\in E(G).$ From the $P_5$ paths  $xu_1wuu_3$, $xu_1uwu_3$ and $xu_1uu_3w$,  $c(xu_1)$ is in $\{1,4\}$, $\{1,5\}$ and $\{4,5\}$ respectively, which is a contradiction.

If both $u$ and $w$ are saturated vertices, then we are done as $G[S]$ is a $5$-vertex and $7$-edge nice-component and we are done.

Without loss of generality assume $xu\in E(G)$. If $xw\notin E(G)$, it can be checked that $d(x)=1$ and we are done. Indeed, from the $P_5$ paths $xuu_3wu_2$ and $xuu_3wu_1$, we get $c(xu)=5$. Suppose for contradiction $y\in N(x)\backslash\{u\}$. Observe that from the $P_5$ paths $yxuu_2w$, $yxuwu_1$ and $yxuwu_2$, $c(yx)$ is respectively in $\{2,3\}$, $\{1,3\}$ and $\{1,2\}$, which is a contradiction. 

On the other hand, if $xw\in E(G)$, then $uwu_1u,uwu_2u,uwu_3u$ and $uwxu$ are four $3$-cycles sharing an edge. In this case we are done by Lemma~\ref{lm5}, as $G[S\cup \{x\}]$ is a $6$-vertex and $9$-edge nice-component.  

\medskip

\noindent \textit{Case 2: The medium-pair is colored with three-coloring scheme.}

\smallskip

Without loss of generality, let $c(uw)=1$, $c(uu_1)=2$, $c(uu_2)=3,\ c(uu_3)=4,\ c(wu_3)=2,\ c(wu_2)=4$ and $c(wu_1)=3$. 

From the proof of Case 1 of Lemma~\ref{ws}, it is easy to see that both $u$ and $w$ are saturated vertices. Moreover, no two vertices in $\{u_1,u_2,u_3\}$ are adjacent. If each vertex in $\{u_1,u_2,u_3\}$ is a saturated, then we are done as $G[S]$ is a $5$-vertex and $7$-edge nice-component. 

Next we show that if there is $x\in V(G)\backslash S$ such that $xu_i\in E(G)$, where $i\in\{1,2,3\}$, then $d(x)=1$ and we are done. To see this, without loss of generality suppose $xu_3\in E(G)$. From the $P_5$ paths $xu_3wu_2u$ and $xu_3wuu_2$, $c(xu_3)\in\{3,4\}$ and $c(xu_3)\in\{1,3\}$ respectively. Thus, $c(xu_3)=3$.

It can be checked that $xu_2, xu_1\notin E(G)$. Indeed, suppose $xu_2\in E(G).$  In this case we get $(u_1uw,wu_2xu_3)$-fish, where edges of the $4$-cycle are colored with at least three colors, namely $2,\ 3$ and $4$. But this is a contradiction to the fish graph coloring scheme stated in Lemma~\ref{c4}. Now suppose $y\in N(x)\backslash \{u_3\}$. From the $P_5$ paths $yxu_3wu_2$, $yxu_3wu$ and $yxu_3uw$, $c(yx)$ is in $\{2,4\}$, $\{1,2\}$ and $\{1,4\}$ respectively, which is again a contradiction. This completes the proof of Lemma~\ref{tr}.  
\end{proof}

\begin{lemma}\label{qw}
Let $u$ and $w$ be two non adjacent vertices in $G$ such that $N(u)\cap N(w)\neq\emptyset$ and $d(u)=4$. If $d(w)=3$, then $N(w)\subset N(u)$. 
\end{lemma}
\begin{proof}
Let $v\in N(u)\cap N(w)$ and suppose for contradiction, $w_1\in N(w)\backslash N(u)$. This implies we have two vertices $u_1$ and $u_2$ in $N(u)\backslash \{v\}$ but not in $N(w)$. Let $u_3$ and $w_2$ be the remaining vertices in $N(u)$ and $N(w)$ respectively. Notice that the vertices could be identical. 

Let $c(uv)=1$ and $c(vw)=2$. For a clear reason, one of the edges in $\{uu_1,uu_2\}$ receives a color not in $\{1,2\}$. Without loss of generality, assume that $c(uu_1)=3$. Moreover, from the degree condition of $u$,  $c(uu^*)=4$, which is a brand new color and $u^*\in\{u_2,u_3\}$. From the $P_5$ paths $w_1wvuu_1$, and $w_1wvuu^*$, we have $c(ww_1)=1$, and as a result $c(ww_2)=3.$ 

If $u_3$ and $w_2$ are distinct vertices, we have a rainbow-$P_5$ path $w_2wvuu^*$, where the edges color from left to right are respectively $3,\ 2,\ 1$ and $4$. 

On the other hand if $u_3$ and $w_2$ are identical vertices, both $c(uw_2)$ and $c(uu_2)$ can not receive $1$ and $3$, as $uu_1$ and $uv$ already received these colors and are incident to the edges $uw_2$ and $uu_2$. Therefore, we get a rainbow-$P_5$ path $w_1ww_2uu_2$, which is a contradiction.  This completes the proof of Lemma~\ref{qw}.
\end{proof}
\begin{remark}\label{rf}
\emph{\begin{enumerate}
\item [(i)] It can be checked that if $d(w)=4$ in the assumptions of Lemma~\ref{qw}, then $N(w)=N(u)$.
\item[(ii)] Observe that $G$ does not contain (as a subgraph) two non adjacent vertices $u$ and $w$ such that $N(u)\cap N(w)\neq\emptyset$, and $d(u)\geq 5$ and $d(w)\geq 3$. Indeed, suppose such a subgraph exists. Let $\{w_1,w_2,w_3\}\subseteq N(w)$ and $w_1\in N(u)$. From the proof of Lemma~\ref{qw}, $w_2, w_3\in N(u)$ and $uw_1w,\ uw_2w$ and $uw_3w$ form a medium-pair. Moreover, form the degree condition of $u$, $u$ is adjacent to at least two vertices not in $\{w_1,w_2,w_3,w\}$.  From Lemma~\ref{ws}, $G$ contains a rainbow-$P_5$ which is a contradiction.
\end{enumerate}}
\end{remark}
\begin{lemma}\label{tn1}
If $G$ contains two adjacent vertices $u$ and $w$ such that $d(u)\geq 5$ and $d(w)\geq 4$, then $G$ contains either a degree-$1$ vertex or two adjacent vertices of degree-$2$ or a nice-component. 
\end{lemma}
\begin{proof} 
If $N(u)\cap N(w)\neq\emptyset$, then from Lemma~\ref{qw}, there are at least three $3$-cycles sharing the edge $uw$, and we can finish the proof using Lemma~\ref{tr}.

Let $\{u_1,u_2,u_3,u_4\}\subseteq N(u)\backslash\{w\}$. First we verify that $d(u_1)\leq 2$. Suppose for contradiction $d(u_1)=3$. One can see that $u_1$ is adjacent to no vertex in $\{u_2,u_3,u_4\}$. Indeed, let $u_1u_2\in E(G)$. Since $u_1w\notin E(G)$, $N(u_1)\cap N(w)\neq\emptyset$ and $d(w)\geq 4$, by Lemma~\ref{qw}, $u_2\in N(w)$ and hence $N(u)\cap N(w)\neq\emptyset$,  which is a contradiction.

Let $x,y\in N(u_1)$. From Lemma~\ref{qw}, $x,y\in N(w)$. Therefore, we get a medium-pair $wuu_1, wxu_1$ and $wyu_1$. Since $d(w)\geq 4$, there is a vertex in $N(w)\backslash \{u,x,y\}$, say $w_1$. Clearly we have a four-coloring scheme for the medium-pair and we have two coloring permutation of the cherries. We check for the coloring $c(wy)=c(u_1x)=1,\  c(wx)=c(u_1y)=2,\  c(wu)=3$ and $c(uu_1)=4$. Similar arguments could be given for the other coloring permutation of the medium-pair. Let $u^*\in N(u)\backslash\{u_1,w\}$. Because of the $P_5$ path $u^*uwxu_1$, $c(u^*u)\in\{1,2\}$. Thus from the color options and since $d(u)\geq 5$ it is easy to find a rainbow-$P_5$ path which is a contradiction. 

Now assume $d(u_1)=2$ and $x\in N(u_1)\backslash\{w\}$. It can be shown that $d(x)\leq 2$, and with this we complete the proof. Indeed, for the reason given above if $x\in N(u)$, then $d(x)=2$ and we get two adjacent degree-$2$ vertices $u_1$ and $x$, and hence we are done. Suppose that $d(x)\geq 3$. Then by Remark~\ref{rf}~(ii), the graph contains a rainbow-$P_5$, which is a contradiction. This completes the proof of Lemma~\ref{tn1}. 
\end{proof}
\section{Avoiding some particular graph structures}\label{fg}
In this section we do some further refinement by avoiding particular structures where we can still finish the proof by induction. The following theorems and results are important to doing so. 
\begin{theorem}\label{nm}
Let $u$ and $w$ be two non adjacent degree-4 vertices in $G$ such that $N(u)\cap N(w)\neq\emptyset$. Then $G$ contains either a degree-$1$ vertex or a nice-component.     
\end{theorem}
\begin{proof}Denote $c(uu_i)=i$, where $i\in\{1,2,3,4\}$. Denote also $S=\{u_1,u_2,u_3,u_4\}$ and $S'=S\cup\{u,w\}$.  

From Remark~\ref{rf}~\text{(i)}, we may assume that $N(u)\cap N(w)=\{u_1,u_2,u_3,u_4\}$. Moreover, from the heavy-pair coloring scheme without loss of generality we may assume that $c(wu_4)=3,\ c(wu_3)=4,\ c(wu_2)=1$ and $c(wu_1)=2$. We have two possible arrangements of the colors as shown in Figure~\ref{sas}.

\begin{figure}[ht]
\centering
\begin{tikzpicture}[scale=0.25]
\draw[ultra thick](0,-6)--(3,0)(0,-6)--(8,0)(0,-6)--(-3,0)(0,-6)--(-8,0)(0,6)--(3,0)(0,6)--(8,0)(0,6)--(-3,0)(0,6)--(-8,0);
\draw[fill=black] (0,-6) circle (12pt);
\draw[fill=black] (0,6) circle (12pt);
\draw[fill=black] (3,0) circle (12pt);
\draw[fill=black] (8,0) circle (12pt);
\draw[fill=black] (-3,0) circle (12pt);
\draw[fill=black] (-8,0) circle (12pt);
\node at (-5.5,-3) {$1$};
\node at (5.5,-3) {$4$};
\node at (-2.5,-3) {$2$};
\node at (2.5,-3) {$3$};
\node at (-5.5,3) {$2$};
\node at (5.5,3) {$3$};
\node at (-2.5,3) {$1$};
\node at (2.5,3) {$4$};
\node at (0,-7.5) {$u$};
\node at (0,7.5) {$w$};
\node at (-9.5,0) {$u_1$};
\node at (-1.5,0) {$u_2$};
\node at (9.5,0) {$u_4$};
\node at (1.5,0) {$u_3$};

\begin{scope}[shift={(25,0)}]
\draw[ultra thick](0,-6)--(3,0)(0,-6)--(8,0)(0,-6)--(-3,0)(0,-6)--(-8,0)(0,6)--(3,0)(0,6)--(8,0)(0,6)--(-3,0)(0,6)--(-8,0);
\draw[fill=black] (0,-6) circle (12pt);
\draw[fill=black] (0,6) circle (12pt);
\draw[fill=black] (3,0) circle (12pt);
\draw[fill=black] (8,0) circle (12pt);
\draw[fill=black] (-3,0) circle (12pt);
\draw[fill=black] (-8,0) circle (12pt);
\node at (0,-7.5) {$u$};
\node at (0,7.5) {$w$};
\node at (-5.5,-3) {$1$};
\node at (5.5,-3) {$4$};
\node at (-2.5,-3) {$3$};
\node at (2.5,-3) {$2$};
\node at (-5.5,3) {$2$};
\node at (5.5,3) {$3$};
\node at (-2.5,3) {$4$};
\node at (2.5,3) {$1$};
\node at (-9.5,0) {$u_1$};
\node at (-1.5,0) {$u_3$};
\node at (9.5,0) {$u_4$};
\node at (1.5,0) {$u_2$};
\end{scope}
\end{tikzpicture}
\label{sas}
\caption{Two possible edge-color arrangements of the heavy-pair.}
\end{figure}
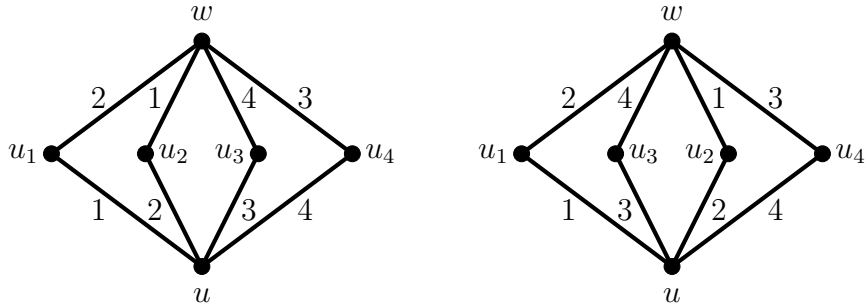

We give the proof for the first color arrangement. But similar argument can be given for the second color arrangement. 

Clearly both $u$ and $w$ are saturated vertices. Let $x\in V(G)\backslash S'$ and $xu_i\in E(G)$, for some $i\in \{1,2,3,4\}$. There is no vertex $y\in V(G)\backslash S'\cup\{x\}$ such that $xy\in E(G)$. Indeed, without loss of generality suppose $xu_1\in E(G)$. From the $P_5$ path $xu_1uu_3w$, $c(xu_1)\in\{3,4\}$. Let $c(xu_1)=3$. Then from the $P_5$ paths $yxu_1uu_2$, $yxu_1uu_4$ and $yxu_1wu_3$, $c(yx)$ is respectively in $\{1,2\}$, $\{1,4\}$ and $\{2,4\}$, which is a contradiction. Similar arguments can be given if we assume that $c(xu_1)=4.$  Therefore, for each vertex $x\in V(G)\backslash S'$ such that $xu_i\in E(G)$, $xu_j\in E(G)$ for some $j\neq i$, where $i,j\in\{1,2,3,4\}$, otherwise we are done. Moreover, each vertex in $S$ is with degree at most $4$. 

If each vertex in $S$ is of degree-$2$, then $G[S']$ is a $6$-vertex and $8$-edge nice-component and we are done. Moreover, only one edge in $\{u_1u_2, u_3u_4\}$ may exist in $G$ and both $u_1u_4, u_2u_3\notin E(G)$. Indeed, in the former case if both exist we get a bow tie with one of its degree-$2$ vertex which is not saturated, which is a contradiction. To verify the later case, suppose $u_2u_3\in E(G)$. Then from the known edge colors of the $P_5$ path $u_2u_3wu_1u$, $c(u_2u_3)=1$, which is again a contradiction as $wu_2$ is already colored with $1$ and yet incident to the edge $u_2u_3$. 

Suppose $u_1u_2\in E(G)$. Clearly $c(u_1u_2)\in\{3,4\}$, and without loss of generality assume $c(u_1u_2)=3$. It can be checked that both $u_3$ and $u_4$ are saturated vertices. To see this, suppose $u_4$ is not saturated vertex and $x\in  V(G)\backslash S'$ such that $xu_4\in E(G)$. From the $P_5$ paths $xu_4wu_2u, xu_4uu_2u_1$ and $xu_4uu_1u_2$, $c(xu_4)$ is in $\{1,2\},\{2,3\}$ and $\{1,3\}$ respectively which is a contradiction.

Next we show that if $u_1$ is not a saturated vertex, it is adjacent to a degree-$1$ vertex. This also holds for $u_2$. To see this, let $xu_1\in E(G)$. Clearly $c(xu_1)=4$ and $xu_4\notin E(G)$, as $u_4$ is a saturated vertex. If $xu_2\in E(G)$, then $c(xu_2)\in \{3,4\}$. But both colors are not possible as $u_1u_2$ and $u_1x$ are colored with $3$ and $4$ and are incident to $xu_2$. Therefore, if such a vertex $x$ exist, then $d(x)=1$, and we are done otherwise $G[S']$ is a $6$-vertex and $9$-edge nice-component and we are done. We distinguish two cases depending on the degrees of vertices in $S$.

\medskip

\noindent \textit{Case 1: There is a degree-$4$ vertex in $S$.}

\smallskip

Without loss of generality assume $d(u_1)=4$ and $z_1, z_2\in V(G)\backslash S'$ such that $z_1u_1, z_2u_1\in E(G)$. From the planarity of $G$ one of the following condition could happen.
\begin{enumerate}
\item $\{z_1,z_2\}\subset N(u_2)$. From Lemma~\ref{qw} it can be shown that each vertex in $\{z_1,z_2,u_3,u_4\}$ is saturated and therefore $G[S'\cup\{z_1,z_2\}]$ is an $8$-vertex and $12$-edge nice-component.
\item $\{z_1,z_2\}\subset N(u_4)$. Again using Lemma~\ref{qw} it can be verified that each vertex in $\{z_1,z_2,u_2,u_3\}$ is saturated and therefore $G[S'\cup\{z_1,z_2\}]$ is an $8$-vertex and $12$-edge nice-component.
\item $z_1\in N(u_2)$ and $z_2\in N(u_4)$ or vice versa. In both cases using Lemma~\ref{qw}, it can be shown that each vertex in $\{z_1,z_1,u_2,u_3,u_4\}$ are saturated and hence $G[S'\cup\{z_1,z_2\}]$ is an $8$-vertex and $12$-edge nice-component.
\end{enumerate}

\medskip

\noindent \textit{Case 2: There is no degree-$4$ vertex in $S$.}

\smallskip

Without loss of generality, assume $d(u_1)=3$ and $z_1\in V(G)\backslash S'$ such that $z_1u_1\in E(G)$. Obviously $c(z_1u_1)\in \{3,4\}$. Let $c(z_1u_1)=3$. Clearly either $z_1u_2\in E(G)$ or $z_1u_4\in E(G)$, but not both due to the the planarity of $G$.

We show the case when $z_1u_2\in E(G)$, and a similarly argument can be given when $z_1u_4\in E(G)$. It is obvious that $c(z_1u_2)=4$, and $z_1$ is a saturated vertex. If both $u_3$ and $u_4$ are saturated vertices, then $G[S'\cup\{z_1\}]$ is a $7$-vertex and $10$-edge nice-component and we are done. Otherwise, we have a cherry $u_3z_2u_4$ in $G$, where $z_2\in V(G)\backslash S'$ and $c(u_3z_1)=1$ and $c(z_1u_4)=2$ or vice versa. In this case $z_2$ is also a saturated vertex and hence $G[S'\cup\{z_1,z_2\}]$ is an $8$-vertex and $12$-edge nice-component and we are done.  This completes the proof of Theorem~\ref{nm}.  
\end{proof}

\begin{definition}
A \emph{double diamond} is two $4$-cycles sharing exactly one vertex, see a double diamond in Figure~\ref{bvx}. We may describe the double diamond in the figure using its vertices as $(u_1xu_2u,uu_3yu_4)$-double diamond.     
\end{definition}

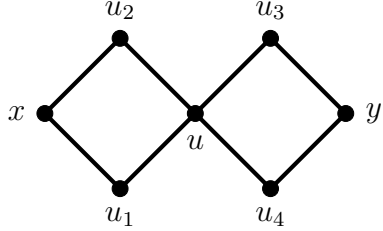
\begin{figure}[ht]
\centering
\begin{tikzpicture}[scale=0.25]
\draw[ultra thick](0,0)--(-4,-4)(-4,4)--(0,0)--(4,4)(4,-4)--(0,0)(4,4)--(8,0)--(4,-4)(-4,4)--(-8,0)--(-4,-4);
\draw[fill=black](-4,-4)circle(12pt);
\draw[fill=black](-4,4)circle(12pt);
\draw[fill=black](4,4)circle(12pt);
\draw[fill=black](4,-4)circle(12pt);
\draw[fill=black](0,0)circle(12pt);
\draw[fill=black](8,0)circle(12pt);
\draw[fill=black](-8,0)circle(12pt);
\node at (-4,-5.5) {$u_1$};
\node at (-4,5.5) {$u_2$};
\node at (4,5.5) {$u_3$};
\node at (4,-5.5) {$u_4$};
\node at (0,-1.5) {$u$};
\node at (9.5,0) {$y$};
\node at (-9.5,0) {$x$};
\end{tikzpicture}
\caption{The double diamond.}
\label{bvx}
\end{figure}

\begin{lemma}\label{fxc}
Let $G$ contains $(u_1xu_2u,uu_3yu_4)$-double diamond and $S=\{u_1,u_2,u_3,u_4,x,y\}$. If there is a vertex $u_1'\in N(u_1)\backslash S$ and $u_2'\in N(u_2)\backslash S$, then $u_1'$ and $u_2'$ are identical vertices. 
\end{lemma}
\begin{proof}
Let the edge-coloring be $c$ and denote $c(uu_i)=i$, where $i\in\{1,2,3,4\}$. It is easy to show that $c(xu_2)\in\{1,3,4\}$. Suppose for contradiction $u_1'$ and $u_2'$ are distinct vertices.  

First, we assume $c(xu_2)=1$. From the $P_5$ path $u_1xu_2uu_3$ and $u_1xu_2uu_4$, $c(xu_1)=2$. Clearly $c(yu_3)\in \{1,2\}$ and $c(yu_4)\in \{1,2\}$. Without loss of generality assume $c(yu_3)=1$ and $c(yu_4)=2$. With these colors we have $c(u_1u_1')=4$ and $c(u_2u_2')=3$. However, we get a contradiction as $u_1'u_1uu_2u_2'$ is a rainbow-$P_5$ path with edge colors from left to right are respectively $4,\ 1,\ 2$ and $3$.

Next assume without loss of generality $c(xu_2)=3$. With this $c(xu_1)=4$. It is easy to show that $c(u_2'u_2)\in \{1,4\}$. Let $c(u_2'u_2)=1$. From the $P_5$ path $u_1'xu_2u_2'$, $c(u_1'u_1)\in\{1,3\}$. It can not be $1$ as $uu_1$ is an incident edge to $u_1'u_1$ and already received this color. Thus, $c(u_1'u_1)=3$. From the $P_5$ path $yu_3uu_2u_2'$ and $yu_3uu_1x$, we have $c(yu_3)=1$. This color and the known colors of the $P_5$ path $u_3yu_4uu_2$ result $c(yu_4)=2$. However we obtain a rainbow path $yu_4uu_1u_1'$, where edge colors from left to right are respectively $2,\ 4,\ 1$ and $3$. A similar argument can be be given if we assume that $c(u_2'u_2)=4$. 
\end{proof}

\begin{remark}\label{fvc}
\emph{Notice that from Lemma~\ref{fxc} if $u_3'\in N(u_3)\backslash S$ and $u_4'\in N(u_4)\backslash S$ exist, then $u_3'$ and $u_4'$ are identical vertices. Moreover, if the identical vertices $u_1'$ and $u_2'$ exist (similarly $u_3'$ and $u_4'$), then using the four-coloring scheme of the medium-pair it  can be seen that both $u_3'$ and $u_4'$ do not exist.}   
\end{remark}

\begin{theorem}\label{cdf}
If $G$ contains a vertex of degree at least $5$, then $G$ contains either a degree-$1$ vertex or two adjacent vertices of degree-$2$ or a nice-component.     
\end{theorem}
\begin{proof} Let $u\in V(G)$ such that $d(u)\geq 5$. From Lemma~\ref{tn1}, we may assume each vertex in $N(u)$ is of degree at most $3$. Let $S=\{u_1,u_2,u_3,u_4,u_5\}\subseteq N(u)$ and denote $c(uu_i)=i$, $i\in\{1,2,\dots ,5\}$. 

If $G[S]$ contains two independent edges, then we get a bow tie and we are done. It can be checked that $0\leq e(G[S])\leq 3$. We finish the proof considering cases on the existence of an edge in $G[S]$.
\medskip

\noindent \textit{Case 1: $e(G[S])\neq 0$.}

\smallskip

Suppose $e(G[S])=3$. Clearly the edges form a $3$-cycle. Without loss of generality, assume $G[S]=\{u_1u_2, u_1u_3, u_2u_3\}$. The vertices, $u_1,u_2$ and $u_3$ are saturated vertices. 

First we show that the vertices $u_4$ (similarly $u_5$) is a degree-$2$ vertex. In doing so if $u_4'\in N(u_4)\backslash \{u\}$, then by Remark~\ref{rf}~(ii), $d(u_4')=2$ and hence $u_4$ and $u_4'$ are adjacent degree-$2$ vertices and we are done.

Suppose not and let  $\{u_4',u_4''\}\subset N(u_4)\backslash \{u\}$ and $\{u_5',u_5''\}\subset N(u_5)\backslash \{u\}$. If $c(u_1u_2)=5$, then it is easy to check that $d(u_4)=2$, which is a contradiction. On the other hand if $c(u_1u_2)\neq 5$, again it is easy to check $d(u_5)=2$, which is again a contradiction. 
\medskip

\noindent \textit{Case 2: $e(G[S])=0$.}

\smallskip
For an obvious reason, we may assume that $d(u_i)=3$ for each $i\in\{1,2,3,4,5\}$. Let $S'=\bigcup_{i=1}^5 N(u_i)$. 

From Remark~\ref{fvc}, $G[S']$ contains no double-diamond with a common vertex $u$. In other words, if two $4$-cycles in $G[S']$ contain $u$, then the cycles share an edge $uu_i$, for some $i\in\{1,2,\dots,5\}$. Moreover, the number of $4$-cycles in $G[S']$ containing $u$ is at most $3$ and  this happens when the $4$-cycles share only three edges in $\{uu_1,uu_2,\dots, uu_5\}$. In other words, at least two edges are not contained in any $4$-cycle containing $u$. Without loss of generality, assume the edges be $uu_4$ and $uu_5$.

Notice that $d(u_4)=d(u_5)=3$. Let $N(u_4)\backslash\{u\}=\{u_4',u_4''\}$ and $N(u_5)\backslash\{u\}=\{u_5',u_5''\}$. Clearly $u_4',\ u_4'',\ u_5'$ and $u_5''$ are distinct vertices. It is easy verify that $c(u_4'u_4)$ and $c(u_4''u_4)$ are in $\{5,r\}$ and $c(u_5'u_5)$ and $c(u_5''u_5)$ are in $\{4,r\}$, where $r\neq 4,5$. Clearly we have a vertex $u^*\in\{u_1,u_2,u_3\}$, such that $c(uu^*)\notin\{4,5,r\}$. Moreover, $d(u^*)=1,$ which is a contradiction. This completes the proof of Theorem~\ref{cdf}.
\end{proof}
\begin{theorem}\label{df}
If $G$ contains two adjacent vertices $u$ and $w$ such that $N(u)\cap N(w)\neq \emptyset $ and $d(u)=d(w)=4$, then $G$ contains either a degree-$1$ vertex or two adjacent vertices of degree-2 or a nice-component.  
\end{theorem}
\begin{proof}
Let $N(u)\backslash \{w\}=\{u_1,u_2,u_3\}$ and $N(w)\backslash\{u\}=\{w_1,w_2,w_3\}$. If $|N(u)\cap N(w)|=3$, then we are done by Lemma~\ref{tr}. Next we distinguish three cases depending on the size of $N(u)\cap N(w)$.
\medskip

\noindent \textit{Case 1: $|N(u)\cap N(w)|=2$.}

\smallskip

Suppose $(u_1, w_1)$ and $(u_2,w_2)$ are identical pair of vertices. Observe that $u_3\notin N(u_1)$ and also $u_3\notin N(u_2)$, otherwise it is easy to find a bow tie with one of its degree-$2$ vertex which is not saturated.

First we show $d(u_3)\leq 2$. Indeed, suppose $d(u_3)=3$. Since $N(u_3)\cap N(w)\neq\emptyset,$ and $d(w)=4$, then $N(u_3)\subset N(w)$. This implies, there is a vertex in $N(u_3)\backslash\{u\}$ adjacent either $u_1$ or $u_2$, but this is a contradiction. 

Let $u_3'\in N(u_3)\backslash\{u\}$. Clearly $u_3'\notin N(u)$. Here we show that $d(u_3')\leq 2$ and finish the proof of this case. Suppose $d(u_3')=3$. Since $N(u)\cap N(u_3')\neq\emptyset$ and $d(u)=4$, then $N(u_3')\subset N(u)$.

Observe that $N(u_3')$ can not be $\{w,u_1\}$ or $\{w,u_2\}$, because in both cases we create a bow tie. Thus, $N(u_3')=\{u_1,u_2\}$. From this we get a medium-pair $u_3'u_3u,\ u_3'u_2u$ and $u_3'u_1u$. Since $d(u)=4$, we must use four-coloring scheme for the medium-pair. For this we have two coloring permutations for the cherries. Let the coloring be $c(uu_1)=c(u_2u_3')=2,\ c(u_1u_3')=c(uu_2)=1,\  c(uu_3)=3$ and $c(u_3u_3')=4$. However this coloring results the $4$-cycle of $(u_1wu,uu_2u_3'u_3)$-fish to receive at least three colors, which is a contradiction to the coloring scheme of a fish graph. Similar argument can be given for the other color permutation of the cherries. 
\medskip

\noindent \textit{Case 2: $|N(u)\cap N(w)|=1$.}

\smallskip

Let $(u_1,w_1)$ be the identical pair of vertices, and $c(u_1w)=1, c(u_1u)=2$ and $c(uw)=3$. Obviously one of the edge in $\{uu_2,uu_3\}$ must receive a brand new color. Let $c(uu_2)=4$. From the known colors of the $P_5$ path $w_2wu_1uu_2$, $c(w_2w)\in\{2,4\}$. From the color options we have and $d(w)=4$, then we may assume $c(ww_2)=4$ and $c(ww_3)=2$. Moreover, it can be seen that $c(uu_3)=1.$ 

Observe that $u_2\notin N(u_3)$, otherwise get a bow tie. Also observe that $u_2w_2\notin E(G)$. Indeed, suppose such an edge exists. From the $P_5$ paths $u_1wuu_2w_2,\ wu_1uu_2w_2$ and $u_1uww_2u_2$, $c(u_2w_2)$ is in $\{1,3\}, \{1,2\}$ and $\{2,3\}$ respectively, and this is a contradiction. In addition to these, it can be seen that $u_2u_1\notin E(G)$. Indeed, suppose such an edge exist. From the $P_5$ path $w_2wuu_1u_2$,\ $c(u_1u_2)\in\{3,4\}$. It can not be $4$ because of the incident edge $uu_2$ which is colored with 4. But this results a rainbow-$P_5$ paths, namely $uu_2u_1ww_3$, where the edge colors from left to right are respectively $4,3,1$ and $2$.

Suppose $x\in V(G)\backslash\{w_3\}$ such that $x\in N(u_2)$. From the $P_5$ paths $xu_2uu_1w,\ xu_2uwu_1$ and $xu_2uww_3$, $c(xu_2)$ is in $\{1,2\}, \{1,3\}$ and $\{2,3\}$, which is a contradiction. 
Thus if $u_2w_3\notin E(G)$, then we are done as $d(u_2)=1$. Notice that from symmetry, what we have so far for $u_2$ hold for $w_2$ too.

To complete the proof, we assume that both $u_2w_3, w_2u_3\in E(G)$. Clearly $c(w_3u_2)=1$ and $c(u_3w_2)=2$. Observe that there is no $x\in V(G)\backslash\{u_1,u_3\}$ such that $xw_3\in E(G)$. Indeed, suppose such a vertex and edge exist. Then from the $P_5$ paths $xw_3u_2uu_1,\ xw_3u_2uw$ and $xw_3wuu_3$, $c(xw_3)$ is respectively in $\{2,4\}$, $\{3,4\}$ and $\{1,3\}$, which is a contradiction. 

Let $x\in V(G)\backslash\{u_3,w_3\}$ such that $xu_1\in E(G)$. From the $P_5$ paths $xu_1wuu_2$ and $xu_1uu_2w_3$, we have $c(xu_1)=4$. Now we can show that $d(x)=1$ and we are done. Indeed, first notice that $xw_3,\ xu_3,\ xu_2,\ xw_2\notin E(G)$. Next, let $y\in V(G)\backslash \{u_1\}$ and $yx\in E(G)$. Remember $y\notin N(u_1)$, otherwise we create a bow tie. From the $P_5$ paths $yxu_1uw,\ yxu_1wu$ and $yxu_1uu_3$, $c(yx)$ is respectively in $\{2,3\}, \{1,3\}$ and $\{1,2\}$, which is a contradiction. 

Observe that if $w_3u_1\in E(G)$, then it can be shown that both $u_3u_1$ and $w_3u_3$ are not in $E(G)$. Indeed, the former is true, as otherwise a bow tie will be created. To show the latter case, suppose note and  $u_3w_3\in E(G)$. From the $P_5$ path $wuu_1w_3u_2$, $c(u_1w_3)=3$. Moreover, from the $P_5$ path $ww_2u_3w_3u_2$, $c(u_3w_3)=4$. But the colors result a rainbow-$P_5$ path $wu_1w_3u_3w_2$, where the edge colors from left to right are respectively $1,3,4$ and $2$.  Therefore if $w_3u_1\in E(G)$, $u_3$ and $w_2$ are adjacent degree-$2$ vertices and we are done. With similar argument if $u_3u_1\in E(G)$, then $u_2$ and $w_3$ are adjacent degree-$2$ vertices and we are done. 

If both $u_3u_1, w_3u_1\notin E(G)$, then we may assume that $w_3u_3\in E(G)$. Otherwise, we have adjacent degree-2 vertices and we are done. If such an edge exists, the graph induced by $\{u,w,u_1,u_2,u_3,w_2,  w_3\}$ is a $7$-vertex and $10$-edge nice-component, see the graph in Figure~\ref{fig:nm}~(left). This completes the proof of Theorem~\ref{df}.
\end{proof}

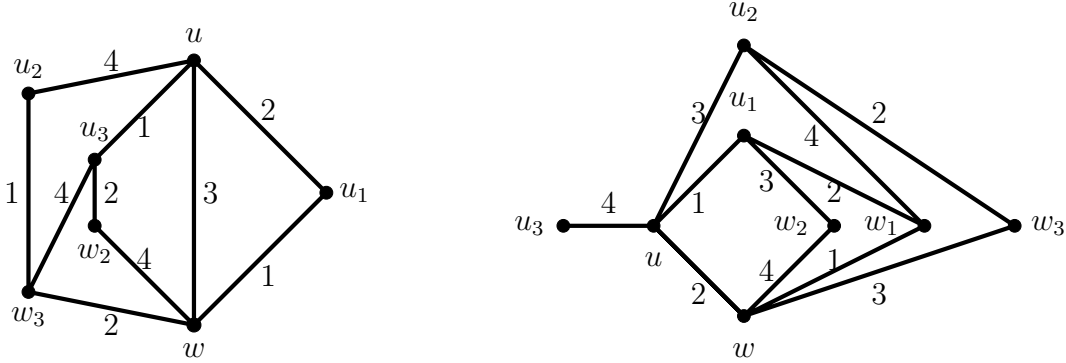
\begin{figure}[ht]
\centering
\begin{tikzpicture}[scale=0.22]
\draw[ultra thick](0,8)--(0,-8)--(8,0)--(0,8)--(-6,2)--(-6,-2)--(0,-8)--(-10,-6)--(-10,6)--(0,8)(-10,-6)--(-6,2);
\draw[fill=black](0,8)circle(11pt);
\draw[fill=black](0,-8)circle(12pt);
\draw[fill=black](8,0)circle(11pt);
\draw[fill=black](-6,2)circle(11pt);
\draw[fill=black](-6,-2)circle(11pt);
\draw[fill=black](-10,6)circle(11pt);
\draw[fill=black](-10,-6)circle(11pt);
\node at (4.5,5) {$2$};
\node at (4.5,-5) {$1$};
\node at (1,0) {$3$};
\node at (-5,0) {$2$};
\node at (-3,4) {$1$};
\node at (-3,-4) {$4$};
\node at (-5,8) {$4$};
\node at (-5,-8) {$2$};
\node at (-11,0) {$1$};
\node at (-8,0) {$4$};
\node at (0,9.5) {$u$};
\node at (0,-9.5) {$w$};
\node at (9.7,0) {$u_1$};
\node at (-6,3.7) {$u_3$};
\node at (-6,-3.7) {$w_2$};
\node at (-10,7.5) {$u_2$};
\node at (-10,-7.5) {$w_3$};
\end{tikzpicture}\qquad\qquad
\begin{tikzpicture}[scale=0.3]
\draw[ultra thick](-4,0)--(0,-4)(-4,0)--(0,4)(-4,0)--(0,8)(0,8)--(12,0)(0,8)--(8,0)(0,4)--(8,0)(0,4)--(4,0)(0,-4)--(-4,0)(0,-4)--(4,0)(0,-4)--(8,0)(0,-4)--(12,0)(-8,0)--(-4,0);
\draw[fill=black](-8,0)circle(8pt);
\draw[fill=black](-4,0)circle(8pt);
\draw[fill=black](4,0)circle(8pt);
\draw[fill=black](8,0)circle(8pt);
\draw[fill=black](12,0)circle(8pt);
\draw[fill=black](0,8)circle(8pt);
\draw[fill=black](0,4)circle(8pt);
\draw[fill=black](0,-4)circle(8pt);
\node at (-2,5) {$3$};
\node at (-2,1) {$1$};
\node at (-2,-3) {$2$};
\node at (-6,1) {$4$};
\node at (1,2) {$3$};
\node at (1,-2) {$4$};
\node at (4,1.5) {$2$};
\node at (4,-1.5) {$1$};
\node at (6,5) {$2$};
\node at (6,-3) {$3$};
\node at (3,4) {$4$};
\node at (-9.5,0) {$u_3$};
\node at (-4,-1.5) {$u$};
\node at (0,-5.5) {$w$};
\node at (2.1,0) {$w_2$};
\node at (6.1,0) {$w_1$};
\node at (13.5,0) {$w_3$};
\node at (0,5.5) {$u_1$};
\node at (0,9.5) {$u_2$};
\end{tikzpicture}
\caption{Possible subgraphs of $G$.}
\label{fig:nm}
\end{figure}

\section{Further refinements and finalizing the proof}\label{cd}
We need the following two theorems to complete our proof. Both theorem gives an estimate of the number of degree-$2$ vertices in the neighbor or second neighbor of a given degree-$4$ vertex.
\begin{theorem}\label{dc}
If $G$ contains two adjacent vertices $u$ and $w$ such that $N(u)\cap N(w)=\emptyset$ and $d(u)=d(w)=4$, then either there is a degree-$1$ vertex or two adjacent degree-2 vertex or there are at least two degree-2 vertices in $N(u)$ and at least two degree-2 vertices in $N(w)$. 
\end{theorem}
\begin{proof}
If two vertices in $N(u)$ or $N(w)$ are adjacent, it is easy to check that the vertices are degree-$2$ vertices and we are done. Theorem~\ref{nm}, we may assume that vertices $N(u)\backslash\{w\}$ and $N(w)\backslash\{u\}$ are with degree at most $3$.

Let $N(u)\backslash\{w\}=\{u_1,u_2,u_3\}$ and $N(w)\backslash\{u\}=\{w_1,w_2,w_3\}$. Suppose $u_1$ and $u_2$ be degree-$3$ vertices. It is enough to show that $G$ contains either a degree-$1$ vertex or two adjacent-vertices of degree-$2$ or a nice-component. 

For clear reasons $N(u_1)\backslash \{u\}\subset \{w_1,w_2,w_3\}$ and $N(u_2)\backslash \{u\}\subset \{w_1,w_2,w_3\}$. Without loss of generality assume $N(u_1)\backslash \{u\}=\{w_1,w_2\}$ and $N(u_2)\backslash\{u\}=\{u_1,u_3\}$. This creates a medium-pair $uww_1,uu_1w_1$ and $uu_2w_1$, see Figure~\ref{fig:nm}~(right). 

Since $d(u)=4$, we use the four-coloring scheme for the medium-pair. Consider the coloring $c(uu_1)=c(ww_1)=1, c(uw)=c(u_1w_1)=2, c(uu_2)=3, c(u_2w_1)=4$ and $c(uu_3)=4$. It can be checked that $c(ww_2)$ and $c(w_2u_1)$ are in $\{3,4\}$, and we may assume $c(u_1w_2)=3$ and $c(ww_2)=4$. From the $P_5$ paths $w_3ww_2u_1u$ and $w_3ww_2u_1w_1$, $c(w_3w)=3$. Again from the $P_5$ path $w_1u_2w_3wu$, $c(u_2w_3)=2$. 

Denote $S=\{u,w,u_1,u_2, w_1,w_2,w_3\}$. Suppose that there is $x\in V(G)\backslash S$ such that $xu_3\in E(G)$. Then from the $P_5$ paths $xu_3uu_1w_1, xu_3uu_1w_2$ and $xu_3uww_3$, $c(xu_3)$ is in $\{1,2\},\ \{1,3\}$ and $\{2,3\}$ respectively, and this is a contradiction. Therefore there is no vertex $x\in V(G)\backslash S$, which is adjacent to $u_3$. But it is still possible that $u_3$ can be adjacent with $w_2$ or $w_3$.  However from the planarity of $G$ either $u_3w_2$ or $u_3w_3$, but not both, is in $G$. If both are not in $G$, then $d(u_3)=1$, and we are done. On the other hand, if one of the edges exists,  then it can be seen that $u_3$ is a saturated vertex  and  $G[S]$ is an $8$-vertex and $12$-edge nice-component, and we are done. This completes the proof of Theorem~\ref{dc}.
\end{proof}
\begin{theorem}\label{gb}
Let $u$ be a degree-$4$ vertex in $G$ such that all its neighbors are with degree at most $3$. Then there is a degree-$1$ vertex or at least two degree-$2$ vertices in $N(u)\cup N_2(u)$. 
\end{theorem}
\begin{proof}
Let $N(u)=\{u_1,u_2,u_3,u_4\}$. We distinguish the the following three cases based on the $e(G[N(u)])$. Remember, we do not have two independent edges in $G[N(u)]$. Otherwise, we $G$ contains a bow tie and we are done. Notice that $0\leq e(G[N(u)])\leq 3$.
\medskip

\noindent \textit{Case 1: $e(G[N(u)])=3$.}

\smallskip

In this case, the edges form $3$-cycle. Considering the degree condition of the neighbors of $u$, $G[N[u]]$ is a $K_4$ incident to an edge. Let the $3$-cycle be $u_1u_2u_3u_1$. 

Notice that $2\leq d(u_4)\leq 3$. Let $x\in V(G)\backslash \{u\}$ such that $xu_4\in E(G)$. Clearly $d(x)\leq 2$. Otherwise, $N(x)\subset \{u_1,u_2,u_3\}$, which is a contradiction as the vertices $u_1,u_2$ and $u_3$ are already saturated from the degree condition of neighbors of $u$. This completes the proof of Case 1. 
\medskip

\noindent \textit{Case 2: $e(G[N(u)])=2$.}

\smallskip

Here the edges must be incident to each other, and let $G[N(u)]=\{u_1u_2,u_1u_3\}$. Clearly, $u_1$ is saturated and $u_4u_3, u_4u_2\notin E(G)$, as a bow tie with unsaturated degree-2 vertex will be created. 

Suppose that $x\in N(u_4)\backslash \{u\}$. One can show that $d(x)\leq 2$. Suppose $d(x)=3$ and $N(x)\backslash \{u_4\}=\{x_2,x_3\}$. By Lemma~\ref{qw} and since $u_1$ is saturated, $\{x_1,x_2\}=\{u_2,u_3\}$. We may assume $(x_2,u_2)$ and $(x_3,u_3)$ are identical pair of vertices.
We have two fish graphs $(u_3u_1u,uu_2xu_4)$-fish and $(u_1u_2u,uu_3xu_4)$-fish. From the coloring scheme of the fish graph, let the first fish graph is colored as $c(uu_1)=1,\ c(uu_3)=3,\  c(uu_2)=c(u_4x)=2$ and $c(uu_4)=c(xu_2)=4$. With this the $4$-cycle of the second fish graph takes at least $3$ colors and this violates the coloring scheme of a fish graph.

Therefore, if $d(u_4)=2$, then its neighbor other than $u$ should be a degree-$2$ vertex, and we are done. Moreover if $d(u_4)=3$, then the two vertices in $N(u_4)\backslash \{u\}$ are degree-$2$ vertices, and again we are done. This completes the proof of Case 2. 
\medskip

\noindent \textit{Case 3: $e(G[N(u)])=1$.}

\smallskip

Let $G[N(u)]=u_1u_2$. First we consider the case that there is $x\in V(G)\backslash\{u\}$ such that  $x\in N(u_3)\cap N(u_4)$. With this we get $(u_1u_2u,uu_3xu_4)$-fish, and from the fish-coloring scheme assume $c(uu_3)=c(xu_4)=3,\ c(uu_4)=c(xu_3)=4,\ c(uu_1)=1,\ c(uu_2)=2$ and $c(u_1u_2)=3$. 

It is easy to check that $u_4$ is a saturated vertex, i.e., $d(u_4)=2$. Moreover, there is no vertex in $y\in V(G)\backslash \{u_1,u_2\}$ such that $yx\in E(G)$. If both $xu_2,xu_1\notin E(G),$ then we are done, as $\{x,u_4\}\subset N(u)\cup N_2(u)$. If both $xu_2,xu_1\in E(G)$, it can be checked that $u_3$ is a saturated vertex, and hence $\{u_3,u_4\}\subset N(u)\cup N_2(u)$.

Without loss of generality, suppose $xu_2\in E(G)$ and $xu_1\notin E(G)$. If $u_2$ is saturated, then we are done, as $\{u_1,u_4\}\subset N(u)\cup N_2(u)$. We assume that there is $y\in V(G)\backslash \{u,u_2\}$ such that $yu_1\in E(G)$. From the $P_5$ paths $yu_1uu_4x$ and $yu_1u_2xu_3$, $c(yu_1)=4$. Now we verify that $d(y)\leq 2$. Indeed, let $y'\in V(G)\backslash \{u_1\}$ so that $y'y\in E(G)$. From the $P_5$ paths $y'yu_1uu_2$ and $y'yu_1u_2x$, $c(y'y)=1$. Therefore, from the color option $d(y)=2$ and moreover $\{y,u_4\}\subset N(u)\cup N_2(u)$ and we are done.   

To complete the proof we assume $N(u_3)\cap N(u_4)=\{u\}$. We may assume that either $u_3$ or $u_4$ is a degree-$3$ vertex, otherwise we are done. 

Let $d(u_4)=3$ and $N(u_4)\backslash\{u\}=\{x,y\}$. We complete the proof by showing either $x$ or $y$ is a degree-$2$ vertex. Indeed, Suppose for $d(x)=d(y)=3$. Considering the degree condition of the neighbors of $u$, it is easy to check that $xy\notin E(G)$.

Let $N(x)\backslash \{u_4\}=\{x_1,x_2\}$ and $N(y)\backslash \{u\}=\{y_1,y_2\}$. Observe that both $\{x_1,x_2\}$ and $\{y_1,y_2\}$ are $\{u_1,u_2\}$, since $N(u_4)\cap N(u_3)=\{u\}$. But this force $d(u_1)=4$ and $d(u_2)=4$, which is a contradiction to the assumption of the degree condition of the neighbors of $u$. This completes the proof of Case 3. 
\medskip

\noindent \textit{Case 4: $e(G[N(u)])=0$.}

\smallskip

Clearly there is a degree-$3$ vertex in $N(u)$, otherwise we are done. Let $d(u_1)=3$ and $x,y\in N(u_1)\backslash \{u\}$. Since $e(G[N(u)])=0$, both $x, y\notin N(u)$. 

Suppose $x$ and $y$ are both degree-$3$ vertices. It can be checked that $xy\notin E(G)$. Let $N(x)\backslash\{u_1\}=\{x_1,x_2\}$ and $N(y)\backslash \{u_1\}=\{y_1,y_2\}$.  Clearly $N(u)=\{x_1,x_2,y_1,y_2\}$. Without loss of generality assume that $(x_1,u_2)$, $(x_2,u_3)$, $(y_1,u_3)$ and $(y_2,u_4)$ are the identical pairs, see Figure~\ref{daq}.

Consider the medium-pair with cherries $uu_1x, uu_2x$ and $uu_3x$. Since $d(u)=4$, we need to use the four-coloring scheme for the cherries. We have to consider two possible coloring permutations to finish the proof. In both cases, what we prove is $u_2$ and $u_4$ are degree-$2$ vertex or otherwise, a vertex in $N(u_2)\backslash \{u,x\}$ and $N(u_4)\backslash\{u,y\}$ is a degree-$1$ vertex. 

Consider the four-coloring scheme given by $c(uu_3)=c(u_1x)=1,\ c(u_3x)=c(uu_1)=2,\ c(uu_2)=3$ and $c(u_2x)=4$, as shown in the Figure~\ref{daq}~(left). Clearly $c(uu_4)=4$ and $c(u_1y)$ and $c(u_3y)$ are $\{3,4\}$. Without loss of generality, assume $c(u_1y)=3$ and $c(yu_3)=4$.  From the $P_5$ path $u_1yu_4uu_3$, $c(yu_4)=1$

Let $u_2$ is not a degree-2 vertex and $z\in N(u_2)\backslash\{u,x\}$. From the $P_5$ paths $zu_2xu_3u$ and $zu_2xu_1y$, $c(zu_2)=1$. Let again $z'\in N(z)\backslash \{u_2\}$. From the planarity of $G$ and the degree condition of neighbors of $u$, $z'$ can not be any of the vertex in $\{x,y,u_3,u_4\}$. From the $P_5$ paths $z'zu_2xu_3$, $z'zu_2uu_1$ and $z'zu_2uu_4$, $c(z'z)$ is in $\{2,4\}$, $\{2,3\}$ and $\{3,4\}$ which is a contradiction. Therefore, $d(z)=1$ and we are done.

Suppose $u_4$ is not a degree-$2$ vertex and let $t\in N(u_4)\backslash\{u,y\}$. From the $P_5$ paths, $tu_4uu_3x$ and $tu_4yu_1u$, $c(tu_4)=2$. Let $t'\in N(t)\backslash \{u_4\}$. From the $P_5$ paths $t'tu_4yu_1, t'tu_4yu_3$ and $t'tu_4uu_2$, $c(t't)$ is in $\{1,3\}, \{1,4\}$ and $\{3,4\}$ respectively,  which is a contradiction. Therefore, $d(t)=1$, and we are done.

The other possible four coloring-scheme of the medium-pair is when $c(uu_1)=c(u_2x)=1, c(uu_2)=c(uu_2)=c(u_1x)=2, c(uu_3)=3$ and $c(u_3x)=4$, see Figure~\ref{daq}~(right). Clearly $c(u_1y)\in\{3,4\}$. Without loss of generality assume $c(u_1y)=3$. From the $P_5$ path $u_2xu_3yu_1$, $c(u_3y)=1$. From the $P_5$ path $u_2uu_4yu_1$, $c(u_4y)=2$.

Let $u_2$ is not a degree-2 vertex and $z\in N(u_2)\backslash\{u,x\}$. From the $P_5$ paths $zu_2xu_3u$ and $zu_2xu_1y$, $c(zu_2)=3$. Let again $z'\in N(z)\backslash \{u_2\}$. From the planarity of $G$ and the degree condition of neighbors of $u$, $z'$ can not be any of the vertex in $\{x,y,u_3,u_4\}$. From the $P_5$ paths $z'zu_2xu_3$, $z'zu_2uu_1$ and $z'zu_2uu_4$, $c(z'z)$ is in $\{1,4\}$, $\{1,2\}$ and $\{2,4\}$ which is a contradiction. Therefore, $d(z)=1$ and we are done.

Suppose $u_4$ is not a degree-$2$ vertex and let $t\in N(u_4)\backslash\{u,y\}$. From the $P_5$ paths, $tu_4yu_3x$ and $tu_4yu_1u$, $c(tu_4)=1$. Let $t'\in N(t)\backslash \{u_4\}$. From the $P_5$ paths $t'tu_4yu_1, t'tu_4uu_2$ and $t'tu_4uu_3$, $c(t't)$ is in $\{2,3\}, \{2,4\}$ and $\{3,4\}$ which is a contradiction. Therefore, $d(t)=1$, and we are done.

To finalize the proof of of Case 4, we may assume that that three vertices in $N(u)$ are of degree-$3$, otherwise we are done. Suppose $u_1, u_2$ and $u_3$ be such vertices. From the previous argument at least one vertex in $N(u_1), N(u_2)$ and $N(u_3)$ are degree at most $2$. If we have a degree-$1$ vertex in one of the three sets, then we are done. Otherwise, we have obviously two degree-$2$ vertex in $N(u_1)\cup N(u_2)\cup N(u_3)$.  This completes the proof of Case 4. 
\end{proof}
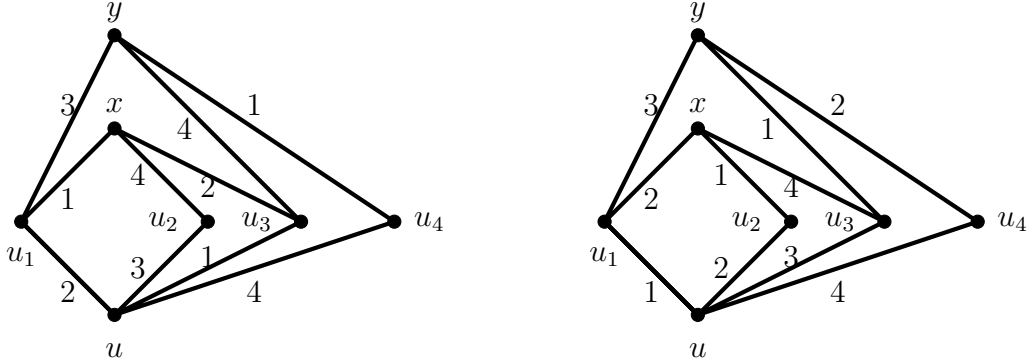
\begin{figure}[ht]
\centering
\begin{tikzpicture}[scale=0.31]
\draw[ultra thick](-4,0)--(0,-4)(-4,0)--(0,4)(-4,0)--(0,8)(0,8)--(12,0)(0,8)--(8,0)(0,4)--(8,0)(0,4)--(4,0)(0,-4)--(-4,0)(0,-4)--(4,0)(0,-4)--(8,0)(0,-4)--(12,0);
\draw[fill=black](-4,0)circle(8pt);
\draw[fill=black](4,0)circle(8pt);
\draw[fill=black](8,0)circle(8pt);
\draw[fill=black](12,0)circle(8pt);
\draw[fill=black](0,8)circle(8pt);
\draw[fill=black](0,4)circle(8pt);
\draw[fill=black](0,-4)circle(8pt);
\node at (-2,5) {$3$};
\node at (-2,1) {$1$};
\node at (-2,-3) {$2$};
\node at (1,2) {$4$};
\node at (1,-2) {$3$};
\node at (4,1.5) {$2$};
\node at (4,-1.5) {$1$};
\node at (6,5) {$1$};
\node at (6,-3) {$4$};
\node at (3,4) {$4$};
\node at (-4,-1.5) {$u_1$};
\node at (0,-5.5) {$u$};
\node at (2.1,0) {$u_2$};
\node at (6.1,0) {$u_3$};
\node at (13.5,0) {$u_4$};
\node at (0,5) {$x$};
\node at (0,9) {$y$};
\end{tikzpicture}\qquad\qquad
\begin{tikzpicture}[scale=0.31]
\draw[ultra thick](-4,0)--(0,-4)(-4,0)--(0,4)(-4,0)--(0,8)(0,8)--(12,0)(0,8)--(8,0)(0,4)--(8,0)(0,4)--(4,0)(0,-4)--(-4,0)(0,-4)--(4,0)(0,-4)--(8,0)(0,-4)--(12,0);
\draw[fill=black](-4,0)circle(8pt);
\draw[fill=black](4,0)circle(8pt);
\draw[fill=black](8,0)circle(8pt);
\draw[fill=black](12,0)circle(8pt);
\draw[fill=black](0,8)circle(8pt);
\draw[fill=black](0,4)circle(8pt);
\draw[fill=black](0,-4)circle(8pt);
\node at (-2,5) {$3$};
\node at (-2,1) {$2$};
\node at (-2,-3) {$1$};
\node at (1,2) {$1$};
\node at (1,-2) {$2$};
\node at (4,1.5) {$4$};
\node at (4,-1.5) {$3$};
\node at (6,5) {$2$};
\node at (6,-3) {$4$};
\node at (3,4) {$1$};
\node at (-4,-1.5) {$u_1$};
\node at (0,-5.5) {$u$};
\node at (2.1,0) {$u_2$};
\node at (6.1,0) {$u_3$};
\node at (13.5,0) {$u_4$};
\node at (0,5) {$x$};
\node at (0,9) {$y$};
\end{tikzpicture}
\caption{Possible subgraphs of $G$.}
\label{daq}
\end{figure}  

In completing the proof, from now on we consider $G$ as a rainbow-$P_5$-free $n$-vertex connected planar containing no degree-$1$ vertex or two adjacent degree-$2$ vertices. With these, we have the following remark.
\begin{remark}\label{lk}
\emph{\begin{enumerate} 
\item[(i)] From Theorem~\ref{cdf}, each vertex in $G$ is of degree at most $4$, i.e., $\Delta(G)\leq 4$.
\item [(ii)] From Theorem~\ref{df}, $G$ contains no two adjacent vertices $u$ and $w$ such that $d(u)=d(w)=4$ and $N(u)\cap N(w)\neq\emptyset$.
\item [(iii)] From Theorem~\ref{nm}, $G$ contains no two non adjacent vertices $u$ and $w$ with $N(u)\cap N(w)\neq\emptyset$ such that $d(u)=d(w)=4$. 
\end{enumerate}}
\end{remark}
We need the following two important observations. 
\begin{enumerate}
\item Let $w_1w_2\in E(G)$, where $d(w_1)=d(w_2)=4$ and $N(w_1)\cup N(w_2)=\emptyset$. From Theorem~\ref{dc}, $N(w_1)$ and $N(w_2)$ contains two degree-$2$ vertices each. It can be seen that for each degree-$2$ vertex $u$ in $N(w_1)$ (similarly for degree-$2$ vertex in $N(w_2)$), there is no degree-$4$ vertex $w_3$, where $w_3\neq w_2$, such that $u\in N_2(w_3)$.
\item Let $v$ be a degree-$2$ vertex  and $v\in N(w_1)\cup N_2(w_1)$, where $d(w_1)=4$ and all its neighbors are with degree at most $3$. It is simple check that there is at most one degree-$4$ vertex $w_2$ such that $v\in N_2(w_2)$.  
\end{enumerate}
Let $x$ be a degree-$4$ vertex and $y$ be a degree-$2$ vertex in $G$. Denote $d_x$ as the number of degree-$2$ vertex in $N(x)\cup N_2(x)$ and denote $e_y$ as the number of degree-$4$ vertex $w$ such that $y\in N(w)\cup N_2(w)$. It can be checked that $d_x\leq 8$ and $e_y\leq 2$. 
\begin{definition} We define a \emph{new-degree} function, $\Tilde{d}_G: V(G)\longrightarrow \mathbb{R}$, as follows:
\begin{equation*}\tilde{d}_G(x)=\begin{cases}
3 \  \ \ \ \ \ \ \ \ \ \ \ \emph{if} \ d_G(x)=3,\\
4-{d_x}/{2} \ \ \  \emph{if} \  d_G(x)=4,\\ 
2+{e_x}/{2} \ \ \ \emph{if} \ d_G(x)=2.
\end{cases}
\end{equation*}
\end{definition}
Clearly $\Tilde{d}_G$ is well-defined and $\Tilde{d}_G(x)\leq 3$ for each vertex $x\in V(G)$. Thus, by the handshaking lemma, $$2e(G)=\sum\limits_{x\in V(G)}\Tilde{d}_G(x)\leq 3n.$$ Therefore, $e(G)\leq {3n}/{2}$ and this completes the proof of Theorem~\ref{t4}.
\section{Conjectures and concluding remarks}
We pose the following conjectures concerning sharp upper bound of the rainbow planar Tur\'an number of $P_6$ and $P_7$.  
\begin{conjecture}\label{cdc1} For all $n$,
$$2n-O(1)\leq\ex_{\p}^*(n,P_6)\leq 2n.$$   
\end{conjecture}
\begin{conjecture}\label{cdc2} For all $n$,
$${5n}/2-O(1)\leq \ex_{\p}^*(n,P_7)\leq {5n}/{2}.$$    
\end{conjecture}
The lower bound in Conjecture~\ref{cdc1} can be realized by taking disjoint copies of the octahedron. The graph is $4$-edge colorable and as a result rainbow-$P_6$-free. On the other hand, the lower bound in Conjecture~\ref{cdc2} can be realized by taking disjoint copies of the icosahedron. It is easy to see the graph is $5$-edge colorable and hence rainbow-$P_7$-free. 


\section*{Acknowledgments}
Gy\H{o}ri’s research was partially supported by the National Research, Development and Innovation Office NKFIH, grants K132696, K116769, and K126853. Martin's research was partially supported by Simons Foundation grant \#709641. Paulos's research was partially supported by  National Research, Development and Innovation Office NKFIH grant K5081, Simons Foundation grant and CEU Global Teaching Fellowship Program grant. 

Paulos would like to thank Gy\H{o}ri, Martin and Iowa State University, Department of Mathematics, for arranging and hosting his research visit during the course of this work.    
\section*{Dedications}
Paulos wants to dedicate this research work to all \textbf{Ethiopian university professors}, who are under \textbf{political repression} by the current \textbf{savage} and \textbf{corrupted regime}.

\end{document}